\documentclass{adamjoucc}
\usepackage{lastpage}
\usepackage{color}
\usepackage{amssymb}
\usepackage{mathtools}
\usepackage{xcolor}
\usepackage{tikz}
\usepackage{enumerate}
\usetikzlibrary{shapes.geometric}
\usepackage{multirow}

\journalyear{2026}
\journalpages{1}{\pageref{LastPage}}
\evenrunninghead{ }
\oddrunninghead{ }
\oddrunninghead{ }
\makeatletter
\def\@journalname{Preprint}
\def\@issn{ }
\def\@journalnumber{}

\makeatother

\makeatletter
\def\moverlay{\mathpalette\mov@rlay}
\def\mov@rlay#1#2{\leavevmode\vtop{%
    \baselineskip\z@skip \lineskiplimit-\maxdimen
    \ialign{\hfil$\m@th#1##$\hfil\cr#2\crcr}}}
\newcommand{\charfusion}[3][\mathord]{
  #1{\ifx#1\mathop\vphantom{#2}\fi
    \mathpalette\mov@rlay{#2\cr#3}
  }
  \ifx#1\mathop\expandafter\displaylimits\fi}
\DeclareRobustCommand\bigop[1]{%
  \mathop{\vphantom{\sum}\mathpalette\bigop@{#1}}\slimits@
}
\newcommand{\bigop@}[2]{%
  \vcenter{%
    \sbox\z@{$#1\sum$}%
    \hbox{\resizebox{\ifx#1\displaystyle.9\fi\dimexpr\ht\z@+\dp\z@}{!}{$\m@th#2$
}}%
  }%
}
\makeatother

\newcommand{\cupdot}{\charfusion[\mathbin]{\cup}{\cdot}}

\DeclareMathOperator{\rank}{rank}
\DeclareMathOperator{\sgn}{sgn}
\DeclareMathOperator{\Det}{det}
\DeclareMathOperator{\stout}{stout}
\DeclareMathOperator{\stin}{stin}

\newcommand{\king}{\mathsf{K}}
\newcommand{\child}{\pmb{\kappa}}
\newcommand{\Metzler}[1]{\mathfrak{M}\left(#1\right)}
\newcommand{\SM}{\mathbf{S}}

\begin{document}

\begin{frontmatter}

\titledata{Autocatalytic Cores in Reaction Networks with Explicit
  Catalysis}{}

\authordata{Richard Golnik}{Bioinformatics Group, Department of Computer
  Science \& Interdisciplinary Center for Bioinformatics Leipzig
  University, H{\"a}rtelstra{\ss}e 16–18, D-04107 Leipzig, \&
  School of Embedded Composite Artificial Intelligence,
  Leipzig University, H{\"a}rtelstra{\ss}e 16–18,
  D-04107 Leipzig, Germany;
  Germany}{richard@bioinf.uni-leipzig.de}{}

\authordata{Thomas Gatter}{Bioinformatics Group, Department of Computer
  Science \& Interdisciplinary Center for Bioinformatics, Leipzig
  University, H{\"a}rtelstra{\ss}e 16–18, D-04107 Leipzig,
  Germany}{thomas@bioinf.uni-leipzig.de}{}
        
\authordata{Peter F.\ Stadler}{Bioinformatics Group,
  Department of Computer Science \&
  Interdisciplinary Center for Bioinformatics \&
  Center for Scalable Data Analytics and Artificial Intelligence 
  Dresden/Leipzig \& 
  School of Embedded Composite Artificial Intelligence,
  Leipzig University, H{\"a}rtelstra{\ss}e 16–18,
  D-04107 Leipzig, Germany;
  Max Planck Institute for Mathematics in the Sciences,
  Inselstra{\ss}e 22, D-04103 Leipzig, Germany;
  Department of Theoretical Chemistry,
  University of Vienna, W{\"a}hringerstra{\ss}e 17,
  A-1090 Wien, Austria;
  Facultad de Ciencias, Universidad National de Colombia;
  Bogot{\'a}, Colombia;
  Santa Fe Institute, 1399 Hyde Park Rd., Santa Fe 
  NM 87501, USA}{studla@bioinf.uni-leipzig.de}{}

\authordata{Nicola Vassena}{Bioinformatics Group, Department of Computer
  Science \& Interdisciplinary Center for Bioinformatics, Leipzig
  University, H{\"a}rtelstra{\ss}e 16–18, D-04107 Leipzig,
  Germany}{nicola.vassena@uni-leipzig.de}{}

\keywords{Chemical Reaction Networks; Autocatalytic Cores; 
		Child-Selection Cores; Semipositivity; Reversible
 		Completion; Fluffle Graphs.}

\msc{}

\begin{abstract}

Autocatalytic cores are the minimal units in reaction networks (RNs)
responsible for the emergence of autocatalysis. In the absence of explicit
catalysis, i.e., when an entity appears both as a reactant and a product in
the same reaction, they are known to be encoded by square submatrices of
the stoichiometric matrix whose columns can be reordered as an irreducible
\emph{child-selection} (CS) matrix with negative diagonal and nonnegative
off-diagonal (Metzler matrix). In the bipartite K\"onig graph representing
the RN, these CS matrices can be identified by
\emph{fluffles}, i.e., strong blocks with an identical number of entity
and reaction vertices, and such that the entity vertices have
out-degree 1 and the reaction vertices have in-degree $1$.

In this contribution, we adapt the concepts derived for autocatalytic cores
to RNs with explicitly catalyzed reactions, which emerge as \emph{digons},
i.e., elementary circuits in the K\"onig graph of length $2$. In this
setting, we confirm that an inspection of the stoichiometric matrix alone
is inconclusive concerning the presence and number of autocatalytic
cores, and that a more delicate algebraic analysis is
required. Nevertheless, this generalization preserves both the graph
representation of autocatalytic cores as fluffles and their matrix
representation as irreducible Metzler CS matrices, although the diagonal is
no longer necessarily strictly negative.

We further introduce the notion of \emph{hard autocatalytic cores}, namely
those that do not yield other autocatalytic cores upon inclusion of the
reverse reactions of their constituting reactions. Finally, we consider the
case of unit stoichiometries (0 and 1) and show that every autocatalytic
core can be constructed as the superposition of at most two elementary
circuits in the König graph. In particular, autocatalytic cores involving
explicitly catalyzed reactions always contain a spanning subgraph
consisting of a single elementary circuit together with a simple
entity-to-reaction chord. Moreover, we identify the essentially unique
example for which at least two circuits are required.

\end{abstract}

\end{frontmatter}

\section{Introduction}

Autocatalysis is a collective property of a reaction network describing its
capability to catalyze its own production. Two fundamentally different
mathematical frameworks have been used to analyze collective autocatalysis:
in the setting of \emph{Reflexively Autocatalytic and Food-generated} (RAF)
networks \cite{Hordijk:04,Hordijk:18}, every reaction is assumed to be
catalyzed by a member of the network. In the setting of chemical reaction
networks (RNs), on the other hand, explicit catalysts for single reactions
have often been excluded, and both catalysis and autocatalysis have been
treated as entirely emergent properties.

While conceptually simple, a widely accepted definition of autocatalysis
has appeared only recently in \cite{Blokhuis:20}, see \cite{Andersen:20x}
for a comparison with some alternatives, in particular \cite{Barenholz:17}.
A common feature of these definitions is that they are concerned with
subnetworks that are ``productive'' in the sense of admitting a flow that
replenishes the subnetwork and adheres to some additional properties
reminiscent of self-maintenance and closure in chemical organization theory
\cite{Dittrich:07}.

A detailed inspection of the recent publications on autocatalytic systems
still shows subtle differences in the definitions of autocatalytic
subnetworks and the related notions of their minimality in terms of
autocatalytic capacity. Moreover, some authors tend to use different
terminology for essentially the same concept or similar terminology for
concepts that turn out to be substantially different. For example, the
expression \emph{autocatalytic core} has been coined in \cite{Blokhuis:20},
but also used for a more restrictive notion of minimality, referring to
principal submatrices in \cite{Blokhuis:25w, Vassena:24a}. Most recently,
several computational approaches \cite{Gagrani:24,Golnik:25z, Golnik:26q}
have become available to identify minimal autocatalytic subsystems of
various types in large RNs.

The purpose of this note is to give a formal account of autocatalytic cores
in reaction networks containing explicitly catalyzed reactions and, in
doing so, providing a self-contained review on the topic of minimality for
autocatalysis.

\section{Reaction Networks}\label{sec:RN}

\paragraph{Basic Notation.} 
A reaction network (RN) is a pair $(X,R)$ of a finite nonempty set of
entitities $x \in X$ and a set of reactions $r\in R$ of the form
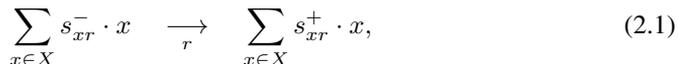
\begin{equation}
  \sum_{x\in X} s^-_{xr} \cdot x \quad \underset{r}{\longrightarrow}
  \quad \sum_{x\in X} s^+_{xr} \cdot x,
\end{equation}
where $s^-_{xr}, s^+_{xr} \ge0$ are the nonnegative \emph{stoichiometric
coefficients} of the molecular count. An entity $x$ is a \emph{reactant} of
$r$ if $s^-_{xr}>0$, a \emph{product} of $r$ if $s^+_{xr}>0$ and a
\emph{catalyst} of $r$ if both $s^-_{xr}>0$ and $s^+_{xr}>0$. In the
presence of a catalyst for a reaction $r$, we say that $r$ is
\emph{explicitly catalyzed}. The \emph{stoichiometric matrix} $\mathbf{S}$
of $(X,R)$ has entries $\mathbf{S}_{xr}\coloneqq s^+_{xr}-s^-_{xr}$. The
first straightforward but crucial observation is that an RN $(X,R)$ is
completely determined by $\mathbf{S}$ if and only if there are no
catalysts. In the presence of catalysts, $\mathbf{S}_{xr}$ describes only
the the \emph{net production ($s^+_{xr}-s^-_{xr}>0$}) or \emph{net
consumption} ($s^+_{xr}-s^-_{xr}<0$) of $x$. We will refer to this as the
\emph{net stoichiometry} of the reaction $r$, in contrast to the
stoichiometry of $r$ given by the coefficients $s^-_{xr}$ and
$s^+_{xr}$. We say that an RN is \emph{well-formed} if for every reaction
$r$ there is a reactant and product in $(X,R)$, i.e, for every $r\in R$,
there is $x_1\in X$ such that $s^-_{x_1r}>0$ and there is a $x_2\in X$ such
that $s^+_{x_2r}>0$.  In particular, for a well-formed RN, production (or
inflow) and degradation (or outflow) reactions, with no reactants and no
products, respectively, are excluded.

\paragraph{Bipartite representation.} An RN can be represented by its
bipartite \emph{K{\"o}nig graph} $\king=(X\cupdot R, E)$ with directed edge
set $E=E_1\cup E_2$, where $E_1\coloneqq \{(x,r)\;|\; s^-_{xr}>0\}$ and
$E_2\coloneqq \{(r,x)\;|\; s^+_{xr}>0\}$. Using terminology from metabolic
networks, where the entities are \emph{metabolites}, we call the edges in
$E_1$ \emph{MR-edges} (Metabolite-Reaction) and edges in $E_2$ \emph{RM-edges}
(Reaction-Metabolite). In this context, the stoichiometric coefficients can
be thought of as weights associated with the respective edges. The property
of being well-formed can then be translated as requiring that each vertex
$r\in R'$ has at least one in-edge (an MR-edge in $E_1$) and at least one
out-edge (an RM-edge in $E_2$).  We remark that the term
``SR-graph'' (Species-Reactions) is also used in the literature for the
same object \cite{CraciunFeinberg:06}.

\paragraph{Subnetworks.} Many of the concepts we will use rely on 
subnetworks, which we formally define as follows.
\begin{definition}
  A \emph{subnetwork} (sub-RN) of $(X,R)$ is a pair $(X',R')$ such that
  $R'\subseteq R$ and $X'\subseteq X$.
\end{definition}

The stoichiometric matrix of a sub-RN is the rectangular submatrix
$\mathbf{S}[X',R']$ obtained from $\mathbf{S}$ by deleting all rows and
columns corresponding to entities not contained in $X'$ and reactions not
contained in $R'$, respectively. Consistent with the above, a sub-RN
$(X',R')$ is \emph{well-formed} if for every reaction $r\in R'$ there is
$x'_1,x'_2\in X'$ such that $s^-_{x'_1r}>0$ and $s^+_{x_2'r}>0$.

In graph-theoretical terms, a sub-RN $(X',R')$ corresponds to the induced
subgraph $\king[X'\cupdot R']$.  We write $X(R')$ for the set of all
reactants $x_1\in X$ and products $x_2\in X$ appearing in the reactions
$r\in R'$:
\begin{equation}
  X(R') \coloneqq \{\,x\in X \mid \exists r\in R' : s^-_{xr}>0 \
  \lor\ s^+_{xr}>0\,\}.
\end{equation}
The induced subgraph $\king[X(R')\cupdot R']$ thus contains the complete
neighborhoods of all its reaction vertices. In turn, the subgraph
$\king[X'\cupdot R']$ contains no \emph{isolated} entity vertices $x\in X$
if and only if $X'\subseteq X(R')$; that is every entity $x\in X'$ is
incident with at least one edge.  Following \cite{Gagrani:24,Kosc:25} we
may define
\begin{equation}
  \begin{split} 
    F(X',R') &\coloneqq
    \{x\in \;(X\setminus X')\;|\;\exists r\in R': s^-_{xr}>0\}, \\
    W(X',R') &\coloneqq 
    \{x\in \;(X\setminus X')\;|\;\exists r\in R': s^+_{xr}>0\}.
    \end{split}
\end{equation}
as the \emph{food} and \emph{waste} of a sub-RN $(X',R')$, respectively.

\paragraph{Child-Selections.} In the context of structural sensitivity
analysis, \emph{child-selections} were first introduced in \cite{Brehm:18}
as injective maps $\kappa$ that assign to each entity $x\in X$ a reaction
$\kappa(x)\in R$ with $s^-_{x\kappa(x)}>0$, i.e. $x$ is a reactant in its
`child' reaction $\kappa(x)$. This approach was later generalized to
subsets $X'\subset X$ and shown to play a key role in determining
structural properties of dynamics of an RN endowed with general kinetic
models, see \cite{VasHunt23, Vassena:24a}. By setting $R'=\kappa(X')$,
$\kappa$ becomes a bijection on the sub-RN $(X',R')$.  We arrived at the
following definition.
\begin{definition}[\cite{Vassena:24a}]\label{def:CS}
  A \emph{child-selection} (CS) $\child$ in $(X,R)$ is a triple
  $\child=(X',R',\kappa)$ such that $(X',R')$ is a sub-RN of $(X,R)$ and
  $\kappa:X'\to R'$ is a bijection such that $\kappa(x)=r$ implies
  $s^-_{xr}>0$, i.e., $x$ is a reactant for reaction $\kappa(x)$.
\end{definition}

By construction, therefore, every reaction in a CS $(X',R',\kappa)$ has a
reactant, and every entity appears at least once as a reactant. In
particular, $|X'|=|R'|$ and hence $\mathbf{S}'\coloneqq
\mathbf{S}[X',R']$ is a \emph{square} submatrix of the stoichiometric matrix
$\mathbf{S}$. We denote by $\mathbf{S}[\child]$ the square matrix obtained by re-ordering
the columns of $\mathbf{S'}$ according to the bijection $\kappa$, i.e. by
setting $\mathbf{S}[\child]_{xy}\coloneqq \mathbf{S}'_{x\kappa(y)}$. We
call such $\mathbf{S}[\child]$ a \emph{CS-matrix}.

A child-selection $\tilde{\child}=(X'',R'',\tilde{\kappa})$ is a
\emph{sub-CS} of $\child=(X',R',\kappa)$ if $X''\subseteq X'$,
$R''\subseteq R'$, and $\tilde{\kappa}:X''\to R''$ is the restriction of
$\kappa$ to $X''$, i.e., $\tilde{\kappa}(x)=\kappa(x)$ for all $x\in
X''$. Clearly, if $\child$ is a CS in $(X,R)$ and $\tilde{\child}$ is a sub-CS of
$\child$, then $\tilde{\child}$ is also a CS in $(X,R)$.

For every child-selection $\child=(X',R',\kappa)$ we define the (not
necessarily induced) bipartite subgraph $\king(\child)$ as follows:
$\king(\child)$ has vertices $V(\king(\child))\coloneqq X'\cupdot R'$ 
and edges $E(\king(\child)) \coloneqq
\{E_1(\child)\cupdot E_2(\child)\}$, where,
\begin{equation}\label{eq:childbipartite}
\begin{split}
  E_1(\child) &\coloneqq
  \{(x,r) \in X'\times R' \text{ such that }r=\kappa(x)\},\\
  E_2(\child) &\coloneqq
  \{(r,x) \in R'\times X' \text{ such that }s^+_{xr}>0\}.
 \end{split}
\end{equation}
Note the subtle difference: $E_2(\child)=E_2(\king)\cap(R'\times X')$,
while $E_1(\child)\subseteq E_1(\king)\cap(X'\times R')$. Therefore, by
construction, $\king(\child)$ is a spanning subgraph of the induced
subgraph $\king[X'\cupdot R']$ with $\king(\child)=\king[X'\cupdot R']$ if
and only if $E_1(\child)=E_1(\king[X'\cupdot R'])$. As noted in
\cite{Golnik:25z}, a subgraph $G$ of $\king$ derives from a CS if and only
if its MR-edges $E_1(G)$ form a \emph{perfect matching}.

\paragraph{Further Notation.} For a vector $\mathbf{v}$ we write
$\mathbf{v}\gg0$ if $\mathbf{v}_i>0$ for all $i$, and $\mathbf{v}>0$ if
$\mathbf{v}_i\ge 0$ for all $i$ and there is at least one $k$ with
$\mathbf{v}_k>0$. A matrix $\mathbf{A}$ is \emph{semipositive} if there is
a positive vector $\mathbf{v}>0$ or, equivalently (making use of the
continuity of linear maps), a strictly positive vector $\mathbf{v}\gg0$,
such that $\mathbf{A}\mathbf{v}\gg 0$, see \cite{Johnson:94}.

\section{Autocatalytic sub-RNs and autocatalytic child-selections}
\label{sec:autocat}

The widely accepted definition of (structural) autocatalysis was proposed
in \cite{Blokhuis:20}.
\begin{definition}\label{def:autSubRN}
  An RN $(X,R)$ is \emph{autocatalytic} if it contains a sub-RN $(X',R')$ such that
\begin{enumerate}[(i)] 
\item $(X',R')$ is  well-formed;
 \item $\mathbf{S}[X',R']$ is semi-positive.
\end{enumerate}
\end{definition}
Sub-RNs satisfying (i) and (ii) are also called \emph{autocatalytic
motifs} in \cite{Kosc:25}. A child-selection $(X',R',\kappa)$ is
autocatalytic if the sub-RN $(X',R')$ satisfies (i) and (ii).

In both \cite{Barenholz:17} and \cite{Blanco:24}, an autocatalytic
subnetwork is in addition required to satisfy \emph{entity autonomy},
i.e., for every $x\in X'$ there are reactions $r',r''\in R'$ such that
$s^-_{xr'}>0$ and $s^+_{xr''}>0$. We shall see below, however, that the
minimality definition automatically implies this property.
\begin{definition}[\cite{Blokhuis:20}]
  An autocatalytic subsystem $(X',R')$ is an \emph{autocatalytic core} if
  it does not contain a strictly smaller autocatalytic sub-RN.
\end{definition}

The setting of \cite{Blokhuis:20} assumes that no catalysts are present in
the RN. An important insight in \cite{Blokhuis:20} is that an autocatalytic
core in a network without catalysts has an associated stoichiometric matrix
$\mathbf{S}[X',R']$ whose columns can be reordered such that the resulting
matrix $\mathbf{A}$ has negative diagonal entries and non-negative
off-diagonal entries. This re-ordering is unique in the sense that for each
entity $x\in X'$ there is a unique reaction $\kappa(x)\in R'$ such that
$\mathbf{S}_{x\kappa(x)}<0$. In the presence of catalysts, however, this result is no longer true. Consider e.g.
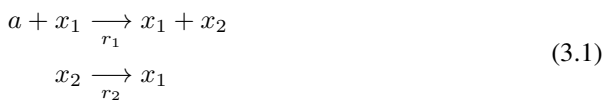
\begin{equation}\label{ex:cat1}
  \begin{split}
      a + x_1 &\underset{r_1}{\longrightarrow} x_1 + x_2  \\
      x_2     &\underset{r_2}{\longrightarrow} x_1      \\
  \end{split}
\end{equation}
The associated sub-RN $(X',R')$ with $X'=\{x_1,x_2\}$ and $R'=\{r_1,r_2\}$
clearly is well-formed with a semipositive stoichiometric matrix
\begin{equation}
  \mathbf{S'} = \left(\begin{matrix}
    0 &  1 \\
    1 & -1 \\
  \end{matrix}\right),
\end{equation}
since $\mathbf{S'}
\begin{pmatrix}2\\1
\end{pmatrix}=
\begin{pmatrix}1\\1
\end{pmatrix} \gg 0$
is a positive vector. $(X',R')$ is thus an autocatalytic sub-RN. Moreover,
one easily checks that no pair of a single entity and a single reaction is
autocatalytic (see also Remark \ref{rmk:singlereaction} below), and hence
$(X',R')$ is an autocatalytic core. However, $\SM'$ has only one negative 
entry, and hence its columns cannot be re-ordered to obtain a matrix with 
negative diagonal entries. Thus, in the presence of catalysis,
we cannot
conclude from the structure of $\mathbf{S'}$ alone that autocatalytic cores
have unique assignments $\kappa:X'\to R'$. We
shall see later (Lemma~\ref{lem:square+kappa}), however, that this property
remains true.

Before we proceed, the following remark considers the case of sub-RNs with
a single entity and reaction $(\{x\},\{r\})$.
\begin{remark}\label{rmk:singlereaction}
A sub-RN $(X',R')$ with a single entity and reaction, $X'=\{x\}$ and
$R'=\{r\}$ is an autocatalytic core if and only if $s_{xr}^+>s_{xr}^->0$,
since only in this case $\mathbf{S}[\{x\},\{r\}]=s_{xr}^+-s_{xr}^->0$, and semipositivity follows:
$\mathbf{S}[\{x\},\{r\}]v>0$ for any $v>0$. In particular, single-entity
autocatalytic cores require the presence of catalysts. In the absence of
catalysts, an autocatalytic core involves at least two species and two
reactions. In the example \eqref{ex:cat1} above, $(X',R')$ does not contain such a
pair, and hence is indeed an inclusion-minimal autocatalytic sub-RN, and
thus an autocatalytic core.
\end{remark}

Let us now turn to the slightly different notion of cores introduced 
in \cite{Vassena:24a} to identify sub-RNs capable of causing dynamical
instability. We paraphrase this definition here for the specific case
of autocatalysis.
\begin{definition} 
  A CS $\child=(X',R',\kappa)$ in $(X,R)$ is an \emph{autocatalytic CS-core} if
  $(X',R')$ is an autocatalytic sub-RN, and there is no proper autocatalytic
  sub-CS $\tilde{\child}$.
\end{definition}
First we note that any autocatalytic pair $(\{x\},\{r\})$ is an
autocatalytic CS-core, since $(\{x\},\{r\},\kappa)$ with $\kappa(x)=r$ is a
CS because $x$ is, by assumption, a reactant of $r$.

More generally, in the networks without catalysts, every autocatalytic core
$(X',R')$ gives rise to a unique CS $\child=(X',R',\kappa)$ where $\kappa$
is defined by the unique negative entry $\mathbf{S}[X',R']_{x\kappa(x)}<0$
in each row (and column). In other words,
$\SM[\child]$ is a semipositive \emph{Metzler
matrix}.  Since there is no autocatalytic sub-RN of $(X',R')$ by
assumption, there is in particular also no autocatalytic sub-CS of
$\child=(X',R',\kappa)$, which therefore is an autocatalytic CS-core.  We
shall see in Sec.~\ref{sec:corescat} below that autocatalytic cores are
always autocatalytic CS-cores also in RNs with catalysts.

The converse is not true, with or without catalysis: there are
autocatalytic CS-cores that are not autocatalytic cores. To see this,
consider the following simple RN:
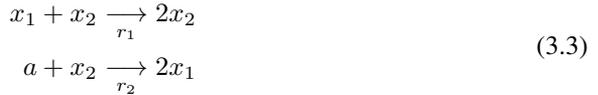
\begin{equation}\label{ex:cat2}
  \begin{split}
    x_1 + x_2 &\underset{r_1}{\longrightarrow} 2x_2  \\
    a + x_2 &\underset{r_2}{\longrightarrow} 2x_1  \\
  \end{split}
\end{equation}
Let $\child=(X',R',\kappa)$ be the CS with $X'=\{x_1,x_2\}$,
$R'=\{r_1,r_2\}$, and $\kappa(x_i)=r_i,i=1,2$. Note that this
is the only CS on $(X',R')$. The stoichiometric matrix
\begin{equation}
  \mathbf{S'} = \left(\begin{matrix}
    -1 &  2 \\
     1 & -1 \\
  \end{matrix}\right)
\end{equation}
is semipositive: for $\mathbf{v}=
{\scriptsize\begin{pmatrix}3\\2\\ \end{pmatrix}}$ we obtain
$\mathbf{S'}\mathbf{v}=
{\scriptsize\begin{pmatrix}1\\1\\ \end{pmatrix}}\gg0$, and hence
$(X',R',\kappa)$ is an autocatalytic CS. Obviously, both sub-CS associated
with $\{x_1,\kappa(x_1)=r_1\}$ and $\{x_2,\kappa(x_2)=r_2\}$ are
non-autocatalytic and thus $\child$ does not contain a smaller
autocatalytic CS and it is an autocatalytic CS-core. The RN
\eqref{ex:cat2}, however, \emph{does} contain another autocatalytic core
$(X'',R'')=(\{x_2\},\{r_1\})$, corresponding to the autocatalytic reaction
$x_2+...\underset{r_1}\rightarrow 2x_2$. Clearly, the associated CS
$\pmb{\lambda}=(\{x_2\},\{r_1\}, \lambda: x_2\mapsto r_1)$ is \emph{not} a
sub-CS of $\child$ since $\kappa(x_2)=r_2\neq r_1=\lambda(x_2)$ but
$(\{x_2\},\{r_1\})$ is a sub-RN of $(X',R')$. Hence, this system has two
autocatalytic CS-cores, corresponding to $\child$ and $\pmb{\lambda}$, but
only a single autocatalytic core $(\{x_2\},\{r_1\})$ w.r.t.\ to set
inclusion.

This difference is not limited to RNs with catalysis. Expanding the
autocatalytic reaction in this example by inserting an intermediate product
$x_3$ yields the following RN:
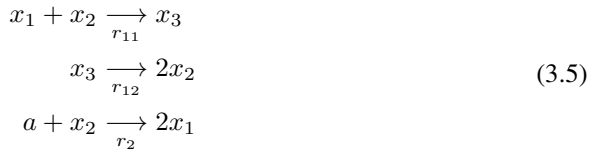
\begin{equation}\label{ex:cat2b}
  \begin{split}
      x_1 + x_2 &\underset{r_{11}}{\longrightarrow}  x_3  \\
      x_3     &\underset{r_{12}}{\longrightarrow} 2x_2  \\
      a + x_2 &\underset{r_2}{\longrightarrow}    2x_1  \\
  \end{split}
\end{equation}
Consider $X'=\{x_1,x_2,x_3\}$ in this order, $R'=\{r_{11},r_{12},r_2\}$,
and the $\kappa$ with $\kappa(x_1)=r_{11}$, $\kappa(x_3)=r_{12}$, and
$\kappa(x_2)=r_2$.  For the child-selection $\child=(X',R',\kappa)$ we have
\begin{equation}\label{eq:autocatCSCoreNonMetzler}
  \mathbf{S[\child]}=
  \begin{pmatrix}
    -1 &  0  &  2 \\ 
     1 & -1  &  0 \\
    -1 &  2  & -1 \\
  \end{pmatrix}
  \qquad \text{and}
  \qquad \mathbf{S}[\child]\begin{pmatrix} 7\\6\\4\end{pmatrix} =
    \begin{pmatrix} 1\\1\\1\end{pmatrix}
\end{equation}
and hence $(X',R',\kappa)$ is an autocatalytic CS. One easily checks that
none of the three sub-CS, whose associated stoichiometric CS-matrix
corresponds to the three principal submatrices of $\mathbf{S}[\child]$ is
autocatalytic since each of them contains a row with only non-positive
entries. Therefore, $\mathbf{S}[\child]$ is an autocatalytic CS-core.

Now consider $X''=\{x_2,x_3\}$, $R''=\{r_{11},r_{12}\}$ and $\lambda:X''\to
R''$ with $\lambda(x_2)=r_{11}$ and $\lambda(x_3)=r_{12}$. Clearly,
$\pmb{\lambda}=(X'',R'',\lambda)$ is a CS.  We have
\begin{equation}
  \mathbf{S}[\pmb{\lambda}] = 
  \begin{pmatrix}
    -1 &  2 \\
     1 & -1 \\
  \end{pmatrix}
  \qquad \text{and}
  \qquad \mathbf{S}[\pmb{\lambda}]\begin{pmatrix} 3\\2\end{pmatrix} =
    \begin{pmatrix} 1\\1\end{pmatrix}
\end{equation}
and hence $(X'',R'',\pmb{\lambda})$ is an autocatalytic CS, and in
particular, also an autocatalytic CS-core. As above, note that
$\kappa(x_3)=r_{12}=\lambda(x_3)$ but $\kappa(x_2)=r_2\ne
r_{11}=\lambda(x_2)$, i.e., $\pmb{\lambda}$ is not a sub-CS of
$\child$. However, $(X'',R'')$ is a sub-RN of $(X',R')$. This implies that
only $(X'',R'')$ is an autocatalytic core.

To better understand the relationships between autocatalytic cores and
autocatalytic CS cores in the general setting of RNs with catalysis, we
analyse their properties in some detail in the following two sections.

\section{Autocatalytic Cores with Catalysts}
\label{sec:corescat}

For networks with catalysts, we generalize the results proved in the SI of
\cite{Blokhuis:20} and \cite{Vassena:24a}, respectively.  We start with a
simple result that follows trivially from the definition of autocatalysis
in RNs without catalysts:
\begin{lemma}
  \label{lem:posentry}
  Let $(X',R')$ be an autocatalytic core. Then for every reaction $r\in
  R'$ there is $x\in X'$ such that $\mathbf{S'}_{xr}>0$.
\end{lemma}
\begin{proof}
  Indirectly, assume that that $\mathbf{S'}_{xr}\le0$ for all $x\in X'$.
  Then the matrix $\mathbf{S''}$ obtained by removing column $r$ from
  $\mathbf{S'}$ is again semipositive, and the sub-RN $(X',R'')$ is
  well-formed, contradicting minimality of $(X',R')$.
\end{proof}
The following statement summarizes the key properties of autocatalytic
cores in general RNs. Its proof follows the arguments in the previous
publications \cite{Blokhuis:20, Vassena:24a} with some subtle differences.

\begin{lemma}
  \label{lem:cs-no-ends} 
  Let $(X',R')$ be an autocatalytic core with $|X'|>1$. Then for every
  $x\in X'$ the following statements hold:
  \begin{itemize}
  \item[(i)]  There is a reaction $r_1\in R'$ such that
    $s_{xr_1}^+ - s_{xr_1}^- >0$, i.e., $\mathbf{S'}_{xr_1}>0$.
  \item[(ii)] There is a reaction $r_2\in R'$ such that
    $s_{xr_2}^->0$ and $s_{xr_2}^+ - s_{xr_2}^-\le 0$, i.e.,
    $\mathbf{S'}_{xr_2}\le0$.
  \item[(iii)] There is a reaction
    $r_x\in R'$ for which $x$ is the only reactant in $X'$, i.e.,\\
    $s^-_{yr_x}>0$ implies $y=x$ for all $y\in X'$.
  \end{itemize}
\end{lemma}
\begin{proof}
  (i) By assumption, $(X',R')$ is autocatalytic and thus
  $\mathbf{S}'=\mathbf{S}[X',R']$ is semipositive: there exists a positive
  vector $\mathbf{v}\gg0$ such that $\mathbf{S'}\mathbf{v}\gg0$. Thus for
  any given $x\in X'$, there exists at least one $r\in R'$ with
  $\mathbf{S'}_{xr}\mathbf{v}_{r}>0$. Consequently,
  $\mathbf{S'}_{xr}=s_{xr}^+ - s_{xr}^- >0$ and (i) holds.

  (ii) Consider the set of entities $N\coloneqq\{x\in X'\;|\;\forall r\in
  R': s^-_{xr}=0\}$ that are not reactants of any reaction and assume that
  $N\neq \emptyset$.  Denote by $\mathbf{S}^{N'}$ the matrix obtained from
  $\mathbf{S'}$ by deleting any subset of rows $N'\subseteq N$. Clearly,
  $\mathbf{S}^{N'}\mathbf{v}$ is still a positive vector, i.e,
  $\mathbf{S}^{N'}$ is semipositive for every choice of $N'$.  For every
  given $r$, there is an entity $y\not \in N$ such that $s^-_{yr}>0$ ((i)
  of Def.~\ref{def:autSubRN}). Denote by $Q\coloneqq \{r\in R'\;|\;\forall
  y\notin N:\mathbf{S}_{yr}\le0 \}$ the set of reactions for which all
  net-stoichiometric products lie in $N$. Three options appear depending on
  whether (a) $Q=\emptyset$; (b) $Q=R'$; (c) $\emptyset\subsetneq
  Q\subsetneq R'$: we will show a contradiction in each instance.  (a) If
  $Q=\emptyset$, then each reaction $r$ has a product in $X\setminus N$,
  and there is an entity $y\not \in N$ such that $s^-_{yr}>0$. Thus
  $(X'\setminus N,R')$ is an autocatalytic subsystem, contradicting the
  minimality of $(X',R')$. (b) If $Q=R'$, then $\sum_{r\in
    R'}\mathbf{S}_{yr}\mathbf{v}_r\le 0$ for all positive vectors
  $\mathbf{v}$ and all $y\notin N$, i.e., $\mathbf{S'}$ is not
  semipositive, contradicting the assumption that $(X',R')$ is
  autocatalytic. (c) If $\emptyset\subsetneq Q\subsetneq R'$, denote by
  $\mathbf{S''}$ the matrix obtain from $\mathbf{S'}$ by removing all
  columns $r\in Q$. Then $0<\sum_{r\in R'} \mathbf{S'}_{yr}\mathbf{v}_r\le
  \sum_{r\in R'\setminus Q} \mathbf{S'}_{yr}\mathbf{v}_r$ for all $y\in
  X\setminus N$. Denoting by $\mathbf{v'}$ the vector obtained from
  $\mathbf{v}$ by removing the reactions indexed by $r\in Q$ we have
  $\mathbf{S'}[X'\setminus N,R'\setminus Q]\mathbf{v'}\gg 0$, i.e.,
  $\mathbf{S'}[X'\setminus N,R'\setminus Q]$ is semipositive.  Moreover,
  for every $r\in R'\setminus Q$ there is $y\in X'\setminus N$ such that
  $s^-_{yr}>0$ and, by construction, there is $z\in X'\setminus N$ such
  that $\mathbf{S''}_{zr}>0$, and hence $s^+_{zr}>0$. Hence $(X'\setminus
  N,R'\setminus Q)$ is an autocatalytic subsystem strictly contained in
  $(X',R')$, contradicting the assumption that $(X',R')$ is an
  autocatalytic core. {Therefore, $N=\emptyset$, and thus} for every $x\in
  X'$ there is a reaction $r''$ such that $s^-_{xr''}>0$.  Finally, if
  $s_{xr''}^+ - s_{xr''}^->0$, i.e., $s_{xr''}^+ > s_{xr''}^-$, then
  $(\{x\},\{r''\})$ is an autocatalytic subsystem strictly contained in
  $(X',R')$, again contradicting the assumption that $(X',R')$ with
  $|X'|\ge 2$ is an autocatalytic core. Thus (ii) holds.

  (iii) Indirectly suppose that there exists one entity $x\in X'$ such that
  for every reaction $r$ with $s^-_{xr}>0$ there is a $x_r\ne x$ such that
  $s^-_{x_r r}>0$. Consider the subsystem $(X'\setminus\{x\},R')$ with
  matrix $\mathbf{S}^x$ obtained from $\mathbf{S'}$ by deleting row
  $x$. Then $\mathbf{S}^x\mathbf{v}$ is still a strictly positive vector
  obtained by removing the entry indexed by $x$ from
  $\mathbf{S'}\mathbf{v}$, and hence $\mathbf{S}^x$ is semipositive.  By
  Lemma~\ref{lem:posentry}, every column of $\mathbf{S'}$ contains at least
  one positive entry.  Suppose $\mathbf{S}'_{xr}>0$ is the only positive
  entry in column $r$. Then, $s^-_{xr}>0$ and $\mathbf{S}_{xr}>0$ implies
  $s^+_{xr}>x^-_{xr}>0$, i.e., $(\{x\},\{r\})$ is an autocatalytic sub-RN,
  contradicting the assumption that $(X',R')$ is an autocatalytic core. For
  every reaction $r$, therefore, there is an $y\in X'$ with
  $\mathbf{S}_{yr}>0$ with $y\ne x$. By construction, we have $x_r\ne x$
  with $s^-_{x_rr}>0$ and thus $(X'\setminus\{x\},R')$ is a strictly
  smaller autocatalytic subsystem of $(X',R')$, contradicting that
  $(X',R')$ is an autocatalytic core. Thus, statement (iii) holds.
\end{proof}

Statements (i) and (ii) in Lemma~\ref{lem:cs-no-ends} imply that
autocatalytic cores satisfy entity autonomy, which some authors require as
part of their definition of autocatalysis \cite{Barenholz:17,Blanco:24}.

The following statement is a variation on \cite[Thm.4.4]{Johnson:94}: An
$n\times m$ matrix $\mathbf{A}$ with $n<m$ is semipositive if and only if
there is a semipositive $m\times m$ square submatrix $\mathbf{A}$.

\begin{lemma} \cite[SI-Lem.6]{Vassena:24a}
  \label{lem:LA}
  Let $\mathbf{A}$ be a semipositive $n\times m$ matrix.  If
  $\rank\mathbf{A}<m$ then there is a column $r$ such that the
  $n\times (m-1)$ matrix $\mathbf{A}^*$ obtained from $\mathbf{A}$ by
  deleting column $r$ is again semipositive.
\end{lemma}

\begin{lemma}
  \label{lem:square+kappa}
  If $(X',R')$ is an autocatalytic core, then $|X'|=|R'|$ and $\mathbf{S'}$
  is invertible. Moreover, for every $x\in X'$ there is a unique reaction
  $\kappa(x)\in R'$ such that $s^-_{x\kappa(x)}>0$.
\end{lemma}
\begin{proof}
  From Lemma~\ref{lem:cs-no-ends} (iii), we immediately conclude that
  $|X'|\le |R'|$. By Lemma~\ref{lem:LA}, there are reactions that can be
  removed from $(X',R')$ unless $|R'|=\rank\mathbf{S'}$. Hence minimality
  w.r.t.\ the number of reactions implies $\rank\mathbf{S'}\le |X'|\le
  |R'|=\rank\mathbf{S'}$. Therefore we have $|X'|=|R'|=\rank\mathbf{S'}$,
  and $\mathbf{S'}$ is an invertible square matrix. Using
  Lemma~\ref{lem:cs-no-ends} (iii) and $|X'|=|R'|$ every reaction $r$ has
  an associated $x\in X'$ such that $x$ is the only reactant of $r$. By
  Lemma~\ref{lem:cs-no-ends} (ii), every $x\in X'$ is reactant of some
  reaction. It follows that $\kappa(x)=r_x$ is a bijection that is uniquely
  determined by the pairs $(x,r)\in X'\times R'$ with $s^-_{xr}>0$.
\end{proof}
\begin{corollary}
  \label{cor:acore-uniqueCS}
  If $(X',R')$ is an autocatalytic core, then there is a unique
  child-selection $\child=(X',R',\kappa)$ on $(X',R')$.
\end{corollary}
As an immediate consequence, we have
\begin{corollary}
  If $(X',R')$ is an autocatalytic core, then $\child=(X',R',\kappa)$ with
  $\kappa$ specified in Lemma~\ref{lem:square+kappa} is an autocatalytic
  CS-core.
\end{corollary}
Finally, the uniqueness of the reactant for each reaction implies that all
off-diagonal entries $s^-_{x\kappa(y)}=0$ for $y\ne x$ and thus
$\SM[\child]_{xy}= \mathbf{S}_{x\kappa(y)}=s^+_{x\kappa(y)}\ge 0$. On the
other hand, we have $\SM[\child]_{xx}=\mathbf{S}_{x\kappa(x)}\le 0$
unless $|X'|=|R'|=1$.
\begin{corollary}\label{cor:diagonal}
  If $(X',R')$ is an autocatalytic core with $|X'|=|R'|\ge 2$. Then
  $\SM[\child]$ is a Metzler matrix with a non-positive diagonal.
\end{corollary}

The following result echoes Prop.~4 in the SI of \cite{Blokhuis:20}.
\begin{lemma}
  \label{lem:Blokhuis4}
  If $(X',R')$ is an autocatalytic core, then the matrix $\SM[\child]$ is
  irreducible.
\end{lemma}
\begin{proof}
  Let $(X',R')$ be an autocatalytic core and suppose $\SM[\child]$ is
  reducible. Then, after simultaneously reordering the rows $X'$ and
  corresponding columns $R'$, it can be written in the form
  $\SM[\child]={\scriptsize \begin{pmatrix} \tilde{\SM} & \mathbf{0}
      \\ \mathbf{X} & \mathbf{Y} \end{pmatrix}}$, where $\tilde{\SM}$ is an
  irreducible matrix. Its rows and columns correspond to a subset
  $X''\subsetneq X'$ and $R''=\kappa(X'')$ are the corresponding reactions.
  In particular, the restriction of the bijection $\tilde{\kappa}$ of
  $\kappa$ to $X''$ is again a bijection $\tilde{\kappa}:X''\to R''$, so
  that $\tilde{\SM}\coloneqq \SM[\tilde{\child}]$ is the CS-matrix of the
  CS $\tilde{\child}$. Take the positive vector $\mathbf{v}\gg0$ such that
  $\SM[\child]\mathbf{v}\gg0$, then the restriction $\tilde{\mathbf{v}}$ of
  $\mathbf{v}$ to $R''$ satisfies
  $(\tilde{\SM}\tilde{\mathbf{v}})_x=(\SM[\child]\mathbf{v})_x>0$ for all
  rows $x\in X''$, and thus $\tilde{\SM}\tilde{\mathbf{v}}\gg0$, i.e.,
  $\tilde{\SM}$ is semipositive. Since $\tilde{\SM}$ is irreducible, each
  column $r$ of $\tilde{\SM}$ has a nonzero off-diagonal entry, which by
  Cor.~\ref{cor:diagonal} applied to $\SM[\child]$ must be positive, say
  $\tilde{\SM}_{xr}>0$, with $x\neq \tilde{\kappa}^{-1}(r)$. In particular,
  $\tilde{\SM}_{xr}=\SM[\child]_{xr}=s^+_{xr}-s^-_{xr}>0$ yields
  $s^+_{xr}>0$. On the other hand, for each column $r$, the diagonal entry
  corresponds to $\tilde{\SM}_{\tilde{\kappa}^{-1}(r)r}$, and thus by
  Def.~\ref{def:CS} of CS, we have $s^-_{\tilde{\kappa}^{-1}(r)r}>0$.
  Therefore, $(X'',R'')$ is an autocatalytic subsystem strictly contained
  in $(X',R')$, contradicting the assumption that $(X',R')$ is an
  autocatalytic core. This in turn implies that $\SM[\child]$ cannot be
  reducible.
\end{proof}

\paragraph{Column orders.} We conclude this section by noting that, in
Def.~\ref{def:autSubRN}, the order of the columns (which depends on the
network labeling) is irrelevant, since semipositivity does not depend on
it. Indeed, if $Av\gg 0$ for some $v>0$, then for any matrix $A_\pi$
obtained from $A$ by permuting its columns according to a permutation
$\pi$, there exists $v_\pi>0$ such that $A_\pi v_\pi \gg 0$. It suffices to
take $v_\pi$ to be the vector obtained from $v$ by applying the same
permutation $\pi$ to its entries. In contrast, \emph{irreducibility} is a
matrix-property that \emph{does} depend on the column ordering and for this
reason Lemma \ref{lem:Blokhuis4} requires the autocatalytic core to be
given in the order of the associated CS.  This difference is crucial in
particular for networks with catalysis. Consider the following RN:
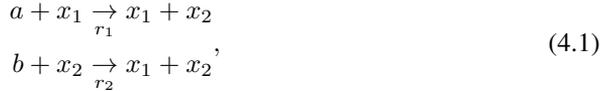
\begin{equation}
  \begin{split}
    a + x_1 &\underset{r_1}{\rightarrow}x_1+x_2\\
    b + x_2 &\underset{r_2}{\rightarrow}x_1+x_2\\
  \end{split},
\end{equation}
and the associated stoichiometric matrix, which is an autocatalytic core
$(\{x_1,x_2\},\{r_1,r_2\})$
\begin{equation}\label{eq:csred}
\SM[\child] = \begin{pmatrix}
    0 & 1\\
    1 & 0\\
\end{pmatrix}.
\end{equation}
Since the matrix is given in CS-form, i.e., the ordering of the columns
follows the CS bijection $\kappa(x_1)=r_1$ and $\kappa(x_2)=r_2$, then
Lemma~\ref{lem:Blokhuis4} applies and the matrix \eqref{eq:csred} is indeed
irreducible. However, for the ordering $\langle r_2,r_1\rangle$ of the
columns, which does not follow from the CS bijection $\kappa$, we obtain
the reducible matrix
\begin{equation}
  \begin{pmatrix}
    1 & 0\\
    0 & 1\\
  \end{pmatrix},
\end{equation}
The related argument in \cite{Blokhuis:20} did not include this subtlety
because (i) it is restricted to networks without catalysis and (ii) the
analogous ``irreducibility'' property was not stated as such in matrix
terms, but rather in terms of strong-connectedness of the associated graph,
see also Sec.~\ref{sec:fluffles}.

\paragraph{Eigenvalues and the determinant of irreducible Metzler CS matrices.} We
conclude this section with two results that will be useful later on. The
first is a consequence of the Perron-Frobenius theorem:
\begin{proposition}
  \label{prop:irMHu=sp}
  Let $\mathbf{A}$ be an irreducible Metzler matrix. Then $\mathbf{A}$ is
  is Hurwitz-unstable if and only if it is semipositive.
\end{proposition}
\begin{proof}
  The matrix $\mathbf{P}\coloneqq \mathbf{A}+c\mathbf{I}$ with
  $c\coloneqq \max_{i\in \{1,\ldots, n\}} |\mathbf{A}_{ii}|\ge 0$ and
  $\mathbf{I}$ identity matrix is non-negative by construction. It is again
  irreducible, has the same eigenvectors of $\mathbf{A}$ and
  $\lambda=\mu+c$ is an eigenvalue of $\mathbf{P}$ if and only if $\mu$ is
  an eigenvalue of $\mathbf{A}$.  The Perron-Frobenius theorem for
  irreducible matrices states that $\mathbf{P}$ has a positive real
  eigenvalue $\lambda^*$ with corresponding strictly positive right
  eigenvector $\mathbf{v}^*$ and left eigenvector $\mathbf{w}^*$ such that
  $\lambda^* > \Re(\lambda)$ for all eigenvalues $\lambda$ of
  $\mathbf{P}$. For $\mathbf{A}$, this yields an eigenvalue
  $\mu^*=\lambda^*-c$ with strictly positive eigenvectors $\mathbf{v}^*$
  (right) and $\mathbf{w}^*$ (left), such that $\mu^*>\Re(\mu)$ for all
  eigenvalues $\mu$ of $\mathbf{A}$.
  First suppose $\mathbf{A}$ is Hurwitz-unstable, i.e.  $\mu^*>0$. For its
  strictly positive right eigenvector $\mathbf{v}^*\gg0$ we get
  $\mathbf{A}\mathbf{v}^*=\mu^*\mathbf{v}^*\gg 0$, i.e., $\mathbf{A}$ is
  semipositive.
  Conversely, assume that $\mathbf{A}$ is semipositive, i.e., there is a
  vector $\mathbf{v}\gg0$ such that $\mathbf{A}\mathbf{v}\gg0$. Take the
  left eigenvector $\mathbf{w}^*$ of $\mu^*$. Since $\mathbf{w}^*\gg 0$, we
  have:
  $0 \ll (\mathbf{w}^*)^T \mathbf{A} \mathbf{v}=\mu^*
  (\mathbf{w}^*)^T\mathbf{v}$. As also $\mathbf{v}\gg0$ by assumption, this
  yields $\mu^*>0$, i.e., $\mathbf{A}$ is Hurwitz-unstable.
\end{proof}

It is well known that an $n\times n$ matrix $\mathbf{A}$ is
Hurwitz-unstable if $(-1)^{n-1}\det\mathbf{A}>0$. The following result,
first shown in \cite{Vassena:24a} for RNs without explicit catalysis, shows
that autocatalytic cores always satisfy this sign condition for $\det\SM[\child]$. We
include a general proof here for completeness.
\begin{lemma}
  \label{lem:descartes}
  \cite{Vassena:24a}
  Let $(X',R')$ be an autocatalytic core with associated $k\times k$ 
  CS-matrix $\SM[\child]$. Then $\SM[\child]$ has one single real-positive 
  eigenvalue and in particular $\sgn\det\SM[\child]=(-1)^{k-1}$.
\end{lemma}
\begin{proof}
  The statement follows from Descartes' rule of signs and the
  Perron-Frobenius theorem. Indeed, the latter implies that $\SM[\child]$
  has at least one real positive eigenvalue, thus Descartes' rule of signs
  yields at least one sign-change in the characteristic polynomial
  $g(\lambda)$ of $\SM[\child]$.  However, minimality prevents any
  sign-change in any coefficient except the last, the constant one, i.e.,
  the determinant. To see this we write the characteristic polynomial as
  follows:
  \begin{equation*}
    \begin{split}
      \det(\lambda \mathbf{I}-\mathbf{A}) =
      \sum_{i=0}^k (-1)^{i} \cdot c_i \cdot \lambda^{k-i}
    \end{split}
  \end{equation*}
  where the $i$-th coefficient, $c_i$, is the sum over all principal minors
  of size $i$. If each principal submatrix of $\SM[\child]$ of size $i$ is
  not Hurwitz-unstable, then $c_i$ has sign $\sgn c_i\in \{0,(-1)^i\}$,
  i.e., $(-1)^i\cdot c_i\geq 0$.  Thus, a sign-change in any coefficient
  other than the $k$-th one would imply the existence of $\tilde{k}\times
  \tilde{k}$ submatrix $\SM[{\tilde{\child}]}$ with a real-positive
  eigenvalue, and thus imply that the corresponding principal minor is
  Hurwitz-unstable.  Consider the (not necessarily unique) principal
  submatrix $\SM[{\tilde{\tilde{\child}}]}$ of $\SM[{\tilde{\child}]}$ such
  that $(-1)^{\tilde{\tilde{k}}-1}\det\SM[{\tilde{\tilde{\child}}]}\ge 0$
  and none of its $l\times l$ principal minors have sign
  $(-1)^{l-1}$. Irreducibility of $\SM[{\tilde{\tilde{\child}}}]$ follows,
  as otherwise there would be a strict principal submatrix with a
  real-positive eigenvalue, a fact that we have just excluded.  Thus,
  Prop.~\ref{prop:irMHu=sp} implies that $\SM[{\tilde{\tilde{\child}}}]$ is
  semipositive, and in particular autocatalytic, contradicting the
  assumption of $(X',R')$ being an autocatalytic core.
\end{proof}

\section{Autocatalytic CS-Cores with Catalysts} 

Some of the properties of CS-cores are already given by their definition as
child-selections: in particular, if $\child=(X',R',\kappa)$ is an
autocatalytic CS core, then $|X'|=|R'|$ and either $|X'|=|R'|=1$ or
$\SM[\child]$ has a non-positive diagonal.

Denote by $\Metzler{\SM}$ the matrix with entries
$\Metzler{\SM}_{xy}=\SM_{xy}$ if $\SM_{xy}\ge 0$ or $x=y$ and
$\Metzler{\SM}_{xy}=0$ otherwise. We call $\Metzler{\SM}$ the \emph{Metzler
part} of $\SM$. Note that $\SM$ is a Metzler matrix if and only if
$\SM=\Metzler{\SM}$.
  
\begin{lemma}
  \label{lem:CSirred}
  Let $\child=(X',R',\kappa)$ be an autocatalytic CS core. Then 
  the Metzler part $\Metzler{\SM[\child]}$ of its CS matrix $\SM[\child]$ 
  is irreducible.
\end{lemma}
\begin{proof}
  We can argue very similarly to the proof of Lemma~\ref{lem:Blokhuis4}.
  Indirectly, suppose $\Metzler{\SM[\child]}$ is reducible. Then,
  after a suitable simultaneous reordering of rows and columns we have
  $\SM[\child]={\scriptsize\begin{pmatrix} \tilde{\SM} & \mathbf{N}
    \\ \mathbf{X} & \mathbf{Y}
  \end{pmatrix}}$, where $\Metzler{\tilde{\SM}}$ is an
  irreducible matrix and $\mathbf{N}$ contains only non-positive
  entries. The rows and columns of $\tilde{\SM}$ correspond to a
  subset $X''\subsetneq X'$ and $R''\coloneqq\tilde{\kappa}(X'')$ are the
  corresponding reactions with $\tilde{\kappa}$ being the restriction of
  the bijection $\kappa$ to $X''$, i.e., $(X'',R'',\tilde{\kappa})$ is a
  sub-CS of $\child$.  By assumption of autocatalyticity, $\SM[\child]$ is
  semipositive, i.e., there is a vector $\mathbf{v}\gg0$ such that
  $\SM\mathbf[\child]{v}\gg0$. Let us write
  $\mathbf{v}=(\tilde{\mathbf{v}},\mathbf{v'})$, where $\tilde{\mathbf{v}}$
  is indexed by $R''$ and $\mathbf{v'}$ is indexed by $R'\setminus
  R''$. Thus $0<(\SM\mathbf[\child]{v})_x =
  (\tilde{\SM}\,\,\mathbf{N})(\tilde{\mathbf{v}},\mathbf{v'})_x =
  (\tilde{\SM}\tilde{\mathbf{v}})_x + \mathbf{N}\mathbf{v'} \leq
  (\tilde{\SM}\tilde{\mathbf{v}})_x$, i.e.,
  $\tilde{\SM}\tilde{\mathbf{v}}\gg0$, and thus $\tilde{\SM}$
  is semipositive. Since $\Metzler{\tilde{\SM}}$ is by indirect
  assumption irreducible,
  each column $r$ of $\tilde{\SM}$ has a positive
  off-diagonal entry, say $\tilde{\SM}_{xr}>0$, with $x\neq
  \tilde{\kappa}^{-1}(r)$. In particular,
  $\tilde{\SM}_{xr}=\mathbf{S}[\child]_{xr}=s^+_{xr}-s^-_{xr}>0$ yields
  $s^+_{xr}>0$. On the other hand, for each column $r$, the diagonal entry
  corresponds to $\tilde{\SM}_{\tilde{\kappa}^{-1}(r)r}$, and thus
  by Def.~\ref{def:CS} of CS, $s^-_{\tilde{\kappa}^{-1}(r)r}>0$.  It
  follows that $(X'',R'',\kappa)$ is an autocatalytic sub-CS of
  $(X',R',\kappa)$ contradicting the assumption that $(X',R')$ is an
  autocatalytic CS-core. Hence $\Metzler{\SM[\child]}$ cannot be reducible.
\end{proof}
An autocatalytic sub-RN with irreducible matrix $\Metzler{\SM}$
trivially satisfies entity autonomy, i.e., every $x\in X$ is produced and
consumed by some reaction. 

Let us now focus on the connection between autocatalytic cores and
autocatalytic CS-cores.  This topic was partially explored in
\cite{Vassena:24a}, where Thm.7.1 shows -- for networks without explicit
autocatalysis -- that autocatalytic cores are characterized as ``unstable
CS-core'' such that $\SM[\child]$ is a Metzler matrix. We clarify here that
the case with explicit catalysis contains a further subtlety, which is
not sufficiently explored in \cite{Vassena:24a}. To connect with the
discussion there, for self-consistency, we first consider RNs without
catalysts:

\begin{lemma}\label{lem:Cscoresmetzler}
  Let $(X,R)$ be an RN without catalysts and let $\child=(X',R',\kappa)$
  be an autocatalytic CS-core such that $\SM[\child]$ is a Metzler
  matrix. Then $(X',R')$ is an autocatalytic core.
\end{lemma}
\begin{proof}
  In the absence of catalysts, reactants appear as negative entries in
  $\SM[\child]$ and hence every reaction has a unique reactant
  $\kappa^{-1}(r)$. Thus the only potentially well-formed subsystems
  $(X'',R'')$ of $(X',R')$ with $|X''|=|R''|$ must satisfy
  $R''=\kappa(X'')$ since otherwise some reactions have no reactant. In
  particular, therefore, every autocatalytic sub-RN $(X'',R'')$ corresponds
  to an autocatalytic sub-CS $(X'',R'',\tilde{\kappa})$, which is excluded
  by assumption.
\end{proof}

Conversely, all autocatalytic cores have a unique CS
$\child=(X',R',\kappa)$ and the corresponding matrix $\SM[\child]$ is a
Metzler matrix. Therefore, we conclude
\begin{proposition}
  \label{prop:7.1}
  Let $(X,R)$ be an RN without catalysts and let $\child=(X',R',\kappa)$ be
  an autocatalytic CS-core. Then $(X',R')$ is an autocatalytic core if and
  only if $\SM[\child]$ is a Metzler matrix.
\end{proposition}
Prop.~\ref{prop:7.1} may be seen as a more explicit version of
\cite[Thm.~7.1]{Vassena:24a} w.r.t. the relationship of autocatalytic
cores and autocatalytic CS-cores, emphasizing the key role of Metzler
matrices in this context.

In case $(X,R)$ contains catalysts, we cannot hope to characterize
autocatalytic cores among autocatalytic CS-core completely in terms of
$\mathbf{S}$. To see this, consider again Example \eqref{ex:cat2}:

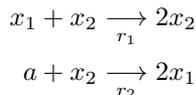
\begin{equation*}
  \begin{split}
      x_1 + x_2 &\underset{r_1}{\longrightarrow} 2x_2  \\
      a + x_2 &\underset{r_2}{\longrightarrow} 2x_1 \\
  \end{split}
\end{equation*}
with the CS $\child=(\{x_1,x_2\},\{r_1,r_2\},\kappa)$, $\kappa(x_1)=r_1$ and
$\kappa(x_2)=r_2$ and associated autocatalytic CS-core
\begin{equation}\label{eq:CScoreNoCore}
  \SM[\child] = \begin{pmatrix}
    -1 &  2 \\
     1 & -1 \\
    \end{pmatrix},
\end{equation}
which, however, and as already commented in Sec.~\ref{sec:autocat}, is
\emph{not} an autocatalytic core due to reaction $r_1$, which is explicitly
autocatalytic in $x_2$ and thus
$\mathbf{S}[\{x_2\},\{r_1\}]=\begin{pmatrix}1\end{pmatrix}$ is the only
autocatalytic core in this example. We note that there are also similar 
examples without autocatalytic reactions, e.g., $a+x_1+x_2 \longrightarrow 
2x_3+x_2$, $x_2 \longrightarrow x_1$, $x_1+x_3  \longrightarrow x_1+x_2$.

Lemma~\ref{lem:square+kappa}, however, provides a suitable alternative
condition that generalizes Prop.~\ref{prop:7.1} and is also
straightforward to verify computationally:
\begin{theorem}
  \label{thm:main}
  Let $(X,R)$ be an RN and let $\child=(X',R',\kappa)$ be an autocatalytic
  CS-core. Then $(X',R')$ is an autocatalytic core if and only if for
  every reaction $r\in R'$, the entity $\kappa^{-1}(r)$ is the only
  reactant of $r$ in $X'$.
\end{theorem}
\begin{proof}
  Lemma~\ref{lem:square+kappa} establishes a necessary condition for
  $(X',R')$ to be an autocatalytic core: the existence, for each reaction
  $r\in R'$, of a unique reactant $x(r)=\kappa^{-1}(r)\in X'$, and thereby
  a unique child-selection determined by the bijection
  $\kappa$. Conversely, if any $r\in R'$ has a unique reactant
  $x(r)=\kappa^{-1}(r)$ in $X'$, then any subset of $(X',R')$ candidate to
  be an autocatalytic core is of the form $(X'',\kappa(X''))$, and thus any
  CS-triple $(X'',R'',\tilde{\kappa})$ is a restriction of
  $\child$. CS-minimality with respect to autocatalysis thus excludes any
  autocatalytic core.
\end{proof}
We note that Prop.~\ref{prop:7.1} follows immediately from
Theorem~\ref{thm:main} if there are no catalysts. 

The subset relationships of autocatalytic cores and autocatalytic CS cores
among the irreducible Metzler CS matrices and CS matrices with irreducible
Metzler part are summarized in Fig.~\ref{fig:CoreRelationships}.

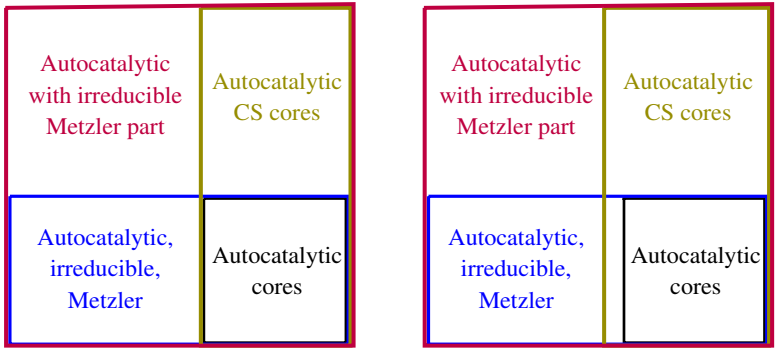
\begin{figure}
  \centering
  \begin{minipage}[c]{0.5\textwidth}
    \centering
    \begin{tikzpicture}
      \node[align=center] (a) at (0,0) {{\small Autocatalytic} \\ {\small cores}};
      \node[label={[xshift=-0.25cm, yshift=-0.75cm, align=center, color=blue] {\small Autocatalytic,} \\ {\small irreducible,}\\ {\small Metzler}}] (a) at (-2,0) {};
      \node[label={[yshift=-0.25cm, align=center, color=olive]{\small Autocatalytic} \\ {\small CS cores}}] (a) at (0,2) {};
      \node[label={[xshift=-0.25cm, yshift=-0.5cm, align=center, color=purple] {\small Autocatalytic} \\  {\small with irreducible}\\ {\small Metzler part}} ] (a) at (-2,2) {};
      \draw[-, black, line width=1.0pt] (-0.96,0.96) to (-0.96,-0.96) to (0.9,-0.96) to (0.9,0.96) to (-0.96,0.96);
      \draw[-, blue, line width=1.2pt] (-3.5,0.99) to (-3.5,-0.98) to (0.93,-0.98) to (0.93,0.99) to (-3.5,0.99);
      \draw[-, olive, line width=1.5pt] (-1,3.5) to (-1,-1) to (0.96,-1) to (0.96,3.5) to (-1,3.5);
      \draw[-, purple, line width=1.5pt] (-3.55,3.5) to (-3.55,-1) to (1,-1) to (1,3.55) to (-3.56,3.5);
    \end{tikzpicture}
  \end{minipage}%
  \hfill%
  \begin{minipage}{0.5\textwidth}
    \begin{tikzpicture}
      \node[align=center] (a) at (0,0) {{\small Autocatalytic} \\ {\small cores}};
      \node[label={[xshift=-0.25cm, yshift=-0.75cm, align=center, color=blue] {\small Autocatalytic,} \\ {\small irreducible,}\\ {\small Metzler}}] (a) at (-2.1,0) {};
      \node[label={[yshift=-0.25cm, align=center, color=olive]{\small Autocatalytic} \\ {\small CS cores}}] (a) at (-0.1,2) {};
      \node[label={[xshift=-0.25cm, yshift=-0.5cm, align=center, color=purple] {\small Autocatalytic} \\  {\small with irreducible}\\ {\small Metzler part}} ] (a) at (-2.1,2) {};
      \draw[-, black, line width=1.0pt] (-0.94,0.96) to (-0.94,-0.96) to (0.9,-0.96) to (0.9,0.96) to (-0.96,0.96);
      \draw[-, blue, line width=1.2pt] (-3.5,0.99) to (-3.5,-0.98) to (0.93,-0.98) to (0.93,0.99) to (-3.5,0.99);
      \draw[-, olive, line width=1.5pt] (-1.2,3.5) to (-1.2,-1) to (0.96,-1) to (0.96,3.5) to (-1.2,3.5);
      \draw[-, purple, line width=1.5pt] (-3.55,3.5) to (-3.55,-1) to (1,-1) to (1,3.55) to (-3.56,3.5);
    \end{tikzpicture}
  \end{minipage}
  \caption{Overview for the set relationship of autocatalytic CS matrices
    in RNs without \textbf{(left)} and with catalysis \textbf{(right)}.  In
    particular, autocatalytic cores are intended in their Metzler CS-
    representation via Cor.~\ref{cor:diagonal}.  The main difference is the
    presence, for explicit catalysis, of autocatalytic CS cores that are
    both irreducible and Metzler, but not autocatalytic cores
    (Ex.~\eqref{ex:cat2b}). Without explicit catalysis, in turn, being
    Metzler identifies autocatalytic cores in the set of autocatalytic CS
    cores (Lemma \ref{lem:Cscoresmetzler}).}
  \label{fig:CoreRelationships}
\end{figure}
  
\section{Strong Blocks, Fluffle Graphs and Algorithmic Considerations}
\label{sec:fluffles}

This section briefly summarizes the algorithmic approach for the detection
of autocatalytic CS-cores from \cite{Golnik:25z,Golnik:26q} and extends its
formal generalities to the case with catalysis. A directed graph $G$ is
called a \emph{strong block} if it is strongly connected and its underlying
undirected graph is 2-connected. In the absence of catalysts,
$\Metzler{\mathbf{S}[\child]}$ is irreducible if and only if the bipartite
subgraph $\king(\child)$ introduced in \eqref{eq:childbipartite} is a
strong block \cite[Thm.~29]{Golnik:25z}. In the presence of catalysis,
however, only one direction holds: the next Lemma shows that having an
irreducible matrix $\Metzler{\mathbf{S}[\child]}$ is a sufficient condition
for $\king(\child)$ to be a strong block. However, the continuation of the
section will clarify that irreducibility of $\Metzler{\mathbf{S}[\child]}$
is not necessary.

\begin{lemma}
  \label{lem:strongblock}
  Let $\child=(X',R',\kappa)$ be a CS whose associated CS matrix
  $\mathbf{S}[\child]$ has irreducible Metzler part
  $\Metzler{\mathbf{S}[\child]}$. Then the subgraph $\king(\child)$ is a
  strong block.
\end{lemma}
\begin{proof}
  From the irreducibility of $\Metzler{\mathbf{S}[\child]}$, we may argue
  as in the proofs of L.27 and L.28 in the Supplement of
  \cite{Golnik:25z}. First observe that for every pair $x,y\in X'$ there is
  a sequence of entity vertices $z_i\in X'$ with $x=z_0$ and $y=z_k$ such
  that $\Metzler{\mathbf{S}[\child]}_{z_{i}z_{i-1}}>0$ where
  $r_{i-1}=\kappa(z_{i-1})$ and $\mathbf{S}_{z_{i}r_{i-1}}>0$. Thus
  $(z_{i-1},\kappa(z_{-1}))\in E_1(\child)$ and $(\kappa(z_{i-1}),z_i)\in
  E_2(\child)$.  Thus
  $(x=z_0,\kappa(z_0),z_1,\kappa(z_1),\dots,\kappa(z_{k-1}),y=z_k)$ is a
  walk in $\king(\child)$.  Moreover, for every reaction vertex $r$ there
  is an edge $(x,r)$ with $x=\kappa^{-1}(r)$ and an edge $(r,y)$ since
  there is $y\in X'$ with $\SM[\child]_{\kappa^{-1}(r)y}>0$. Hence all
  vertices of $\king(\child)$ are reachable from each other and thus
  $\king(\child)$ is strongly connected. Now suppose $\king(\child)$ is
  strongly connected but not a strong block, then there is a cut vertex
  $v\in V(\king(\child))$. Then $v$ has at least two in-edges and two
  out-edges: If $v$ is reaction vertex, $v=r$, its in-edges in
  $\king(\child)$ are of the form $(\kappa^{-1}(r),r)$.  Since $\kappa$ is
  a bijection, there is at most one such edge, and thus we reach a
  contradiction. If $v$ is an entity vertex, i.e. $v=x$, then all its
  out-edges are of the form $(x,\kappa(x))$, i.e., again there is exactly
  one such edge, and thus we reach analogously a contradiction. Thus
  $\king(\child)$ cannot contain a cut vertex, and hence must be a strong
  block.
\end{proof}

Combining Lemmas~\ref{lem:CSirred} and \ref{lem:strongblock} immediately
yields;
\begin{corollary}
  Let $\child=(X',R',\kappa)$ be an autocatalytic CS core. Then the
  subgraph $\king(\child)$ is a strong block.
\end{corollary}

The subgraphs $\king(\child)$ corresponding to irreducible autocatalytic
child-selections $\child$ in RNs without catalysts served as a motivation
for introducing the following class of directed bipartite graphs in
\cite{Golnik:25z}:
\begin{definition}[Fluffle]  
  A directed bipartite graph $\Gamma=(V,E)$ with $V=X\cup R$ is a \emph{fluffle} if
  \begin{enumerate}
  \item $|X|=|R|$,
  \item $\Gamma$ is a strong block,
  \item every $x\in X$ has out-degree $1$ and every $r\in R$ has indegree
    $1$.
  \end{enumerate}  
\end{definition}

This close connection between irreducible child-selections and fluffles
persists in RNs with catalysts. More precisely, we have:
\begin{lemma}
  \label{lem:fufflesufficient} 
  Let $\child$ be a child-selection such that $\Metzler{\mathbf{S}[\child]}$
  is an irreducible matrix. Then $\king(\child)$ is a fluffle.
\end{lemma}
\begin{proof}
  $\king(\child)$ is a subgraph of the bipartite graph
  $\king$. By construction of $\king(\child)$, $|X'|=|R'|$ and every
  reaction $r\in R'$ has a single reactant $x=\kappa^{-1}(r)$, and thus
  indegree $1$, and every entity $x\in X'$ has a single reaction
  $\kappa(x)$ for which it serves as the sole reactant. Thus $x$ has
  outdegree $1$. Finally, by assumption $\Metzler{\mathbf{S}[\child]}$ is
  irreducible and thus $\king(\child)$ is a strong block by
  Lemma~\ref{lem:strongblock}. Thus $\king(\child)$ is a fluffle.  
\end{proof}

In the absence of catalysts, the converse is also true
  \cite[Prop.~32]{Golnik:25z}. In RNs with catalysts, however, a fluffle
  $\king(\child)$ may also give rise to a reducible matrix
  $\SM[\child]$. To see this, consider the following RN:
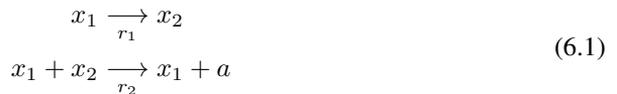
\begin{equation}
  \begin{split}
    x_1 &\underset{r_1}{\longrightarrow} x_2 \\
    x_1 + x_2 &\underset{r_2}{\longrightarrow} x_1 + a		
  \end{split}
\end{equation}
with $X'\coloneqq\{x_1, x_2\}$, $R'\coloneqq \{r_1, r_2\}$, 
child-selection $\child=(X',R',\kappa)$ given by
$\kappa(x_i)\coloneqq r_i$, for $i=1,2$. Then $\king(\child)$, shown in
Fig.~\ref{fig:redFluf} is a fluffle since it is an elementary circuit.  
The corresponding child-selection matrix, however, is 
\begin{equation}
  \SM[\child]=\begin{pmatrix} -1 & 0    \\ 1 & -1   \\ \end{pmatrix},
\end{equation}
which is a reducible Metzler matrix. 

\begin{figure}
  \centering
  \begin{tikzpicture}
    \node[draw, circle] (x_1) at (0,0) {$x_1$};
    \node[draw, circle] (x_2) at (2,2) {$x_2$};
    \node[draw, rectangle] (r_1) at (2,0) {$r_1$};
    \node[draw, rectangle] (r_2) at (0,2) {$r_2$};
    \draw[->, line width=1.0pt] (x_1) to (r_1);
    \draw[->, line width=1.0pt] (r_1) to (x_2);
    \draw[->, line width=1.0pt] (x_2) to (r_2);
    \draw[->, line width=1.0pt] (r_2) to (x_1);
  \end{tikzpicture}
  \caption{Fluffle $\king(\child)$ for a child-selection $\child$ 
    with reducible CS-Metzler matrix.} 
  \label{fig:redFluf}
\end{figure}
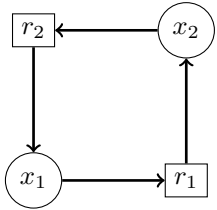

Uniqueness of the reactant for each reaction in an autocatalytic core (see
Lemma~\ref{lem:cs-no-ends} (iii)) implies that the fluffles associated with
autocatalytic cores, i.e., defined by the unique CS $\child$ in
Cor.~\ref{cor:acore-uniqueCS}, are induced subgraphs of $\king$. In the
following we will refer to them as \emph{induced fluffles} (see also
Sect.~\emph{Metzler Matrices and Induced Fluffles} in \cite{Golnik:25z}).
Induced fluffles obviously give rise to Metzler matrices since by
construction each entity is reactant of exactly one reaction. Taken
together, these observations extend the correspondence of Metzler
CS-matrices $\mathbf{S}[\child]$ and induced fluffles $\king[\child]$ in
RNs without catalysts \cite[Cor.47]{Golnik:25z} in the following way:
\begin{corollary}
  \label{cor:ac-inducedfluffle}
  If $(X',R')$ is an autocatalytic core with CS $\child=(X',R',\kappa)$,
  then $\king(\child)$ is the induced graph $\king(\child)=\king[X'\cupdot
    R']$ and $\mathbf{S}[\child]$ is a Metzler matrix.
\end{corollary}
As a consequence, only induced fluffles are candidates for autocatalytic
cores, while fluffles in general may give rise to autocatalytic
CS-cores. Fluffles also have a simple characterization in terms of
so-called (directed) ear decompositions \cite{Grotschel:79}
\begin{proposition} \cite[Thm.33]{Golnik:25z}
  A graph G is a fluffle if and only if it is bipartite with vertex set
  $X\cupdot R$ and it has an ear decomposition such that every ear
  initiates in a reaction vertex $r\in R$ and terminates in a substrate
  vertex $x\in X$.
\end{proposition}

In particular, every elementary circuit $C_i$ in $\king$ is a fluffle and
any fluffle subgraph $G$ in $\king$ can also be constructed by successively
superimposing a sequence $(C_1,C_2,\dots,C_k)$ of elementary circuits such
that $G_1=C_1$, $G_i=G_{i-1}\cup C_i$ for $1\le i\le k$, and $G=G_k$, while
each $G_i$ is again a fluffle. Then $G_i\setminus G_{i-1}$ is a disjoint
union of ears. Since we are mostly interested in maximal fluffles
determined by the set $E_1=\bigcup_{j=1}^i E_1(C_i)$ of MR-edges and the
induced RM-edges $E_2(\king[V(G_i)]$ it suffices to consider superposition
of elementary circuits such that $V(G_{i-1})\subsetneq V(G_i)$, i.e.,
$V(G_i)\setminus V(G_{i-1})\ne\emptyset$. This observation is used already
in \cite{Golnik:25z} to reduce the efforts for enumerating maximal fluffles
as candidates for autocatalytic strong blocks.

As a consequence of Lemma~\ref{lem:fufflesufficient} the maximal fluffle
subgraphs in $\king$ contain all autocatalytic cores as well as all
autocatalytic CS-cores. A na{\"\i}ve strategy therefore is to
\begin{itemize}
\item[(1)] enumerate the set $\mathfrak{F}$ of fluffles in $\king$
\item[(2)] identify the subset $\mathfrak{A}\subseteq\mathfrak{F}$ of
  fluffles for which $\mathbf{S}[\child]$ is positive semidefinite,
  where $\child$ is the unique CS defined by the MR-edges of the fluffle,
\item[(3)] construct the Hasse diagram with respect to submatrices or
  sub-CS, respectively, and 
\item[(4)] retain only the minimal elements.
\end{itemize}
While correct, this approach is too demanding in terms of computational
resources for large networks. It is feasible for small networks such as
models of the core carbon metabolism, or for partial enumeration efforts that
restrict the size of elementary circuits and maximal fluffles, see the Python implementation \texttt{autogato}
\cite{Golnik:25z}.

A much more efficient algorithmic approach can be taken if only
autocatalytic cores are to be computed. The key observation is that in this
case, only induced fluffles have to be considered. Induced fluffles in turn
are superpositions of elementary circuits $C$ that themselves form induced
fluffles, i.e., that do not have chords of the form $(x,r)$ with $x\in
X\cap V(C)$, $r\in R\cap V(C)$ and $(x,r)\in E(\king)\setminus E(C)$. We
call such an elementary circuit $C$ a \emph{Metzler circuit} in $\king$. In
\cite{Golnik:26q} we describe an algorithm to generate Metzler circuits
efficiently. Using the theoretical results of \cite{Golnik:25z} for RNs
without catalysis, it is not difficult to show that all autocatalytic cores
are superpositions of Metzler circuits. As an immediate consequence of
Cor.~\ref{cor:ac-inducedfluffle}, we observe that an analogous result
remains true in RNs with catalysts:

\begin{proposition}
  If $G$ is a subgraph of $\king$ that forms an autocatalytic core then $G$
  is the superposition of Metzler circuits in $\king$.
\end{proposition}
\begin{proof}
  If $G$ is an autocatalytic core, then by Cor.~\ref{cor:ac-inducedfluffle}
  it induces a fluffle in $\king$, and hence it has a representation as
  a superposition of elementary circuits.  All MR-edges in the induced
  subgraph must be located along one of the constituent circuits. Using
  that $\king[V(G)]$ is a fluffle implies that every MR-edge in
  $\king[V(G)]$ must lie in $G$. If $(x,r)\in E(C_1)$ were a chord in some
  other constituent circuit $C_2$, in the superposition $C_1\cup C_2$ then $r$
  has a predecessor $y\ne x$ in $C_2$, contradicting that $C_1\cup C_2$
  and thereby also $G$ is a fluffle.
\end{proof}

Hence, a Metzler circuit $C$ is either itself an autocatalytic core or it
can be combined with an overlapping non-autocatalytic Metzler circuit $C'$.
If this superposition $C\cup C'$ is an induced fluffle then it is again a
candidate for an autocatalytic core. This basic idea is used
\cite{Golnik:26q} to devise a highly efficient algorithm (\texttt{autogatito}) for enumerating
autocatalytic cores. The proposition above ensures that this algorithm
remains correct also in RNs with catalysts.

For the enumeration of autocatalytic CS-cores, it might be helpful to first
determine all autocatalytic cores using \texttt{autogatito} since each of
them is also an autocatalytic CS-core. Then it suffices to find
autocatalytic CS-cores that are not also autocatalytic cores; let us call
these \emph{extra CS cores} for short.  Recall that each maximal fluffle
$G$ is uniquely determined by its matching $E_1(G)=\{(x,\kappa(x))| x\in
V(G)\cap X\}$, i.e., by the underlying child-selection. Moreover, a maximal
fluffle $G_1$ is a subgraph of a maximal fluffle $G_2$ if and only if
$E_1(G_1)\subseteq E_1(G)$ \cite{Golnik:25z}. A fluffle $H\in\mathfrak{F}$
thus is a candidate for being an extra CS core if (i) there is an
autocatalytic core $G$ such that $X(G)\subseteq X(H)$, $R(G)\subseteq R(H)$
and $E_1(G)\not\subseteq E_1(H)$ and (ii) there is no autocatalytic core or
extra-CS-core $F$ with $E_1(F)\subsetneq E_1(H)$. How to utilize this
observation for the construction of an algorithm is, however, not obvious.
Hence we leave this topic as an open direction for future research.

\section{Irreversibility of Autocatalytic Cores}

In previous work, it has been noted that in a setting where each 
reaction $r$ is assumed to be reversible, the directionality of an 
autocatalytic core is, anyway, uniquely determined
by its stoichiometry, see e.g.\ SI of \cite{Kosc:25}. In this section, we
give a complete account of this phenomenon in our setting.
\begin{definition}
  The inverse $\bar{r}$ of a reaction $r\in R$ is the reaction given
  by $s^-_{x\bar{r}}\coloneqq s^+_{xr}$ and $s^+_{x\bar{r}}\coloneqq s^-_{xr}$ 
  for all $x\in X$.
\end{definition}
It follows immediately that the columns in the stoichiometric matrix
belonging to a reaction $r$ and its inverse $\bar{r}$ satisfy
$\mathbf{S}[X,\{\bar{r}\}]= -\mathbf{S}[X,\{r\}]$.  Given a subset
$R'\subseteq R$ we write $\bar{R}' \coloneqq\{\bar{r}\;|r\;\in R'\}$ for
the corresponding set of inverse reactions and call $(X',R'\cup\bar{R}')$
the \emph{reversible extension} of the subgraph $(X',R')$. To avoid an
overload of notation we denote by $\SM$ the stoichiometric matrix for both
networks, as clear in the context.

In this section, we will be concerned with the structure of the reversible
extension $(X',R'\cup \bar{R}')$ of a given autocatalytic core
$(X',R')$. Note that an autocatalytic core $(X',R')$ cannot contain both a reaction $r$ and its inverse $\bar{r}$, since its stoichiometric matrix $\mathbf{S}'$ is invertible via Lemma \eqref{lem:square+kappa}.

For our discussion below it will be useful to distinguish elementary
circuits in the bipartite K{\"o}nig graph's graph $\king$ in terms of their
stoichiometric coefficients. Our terminology is inspired by
\cite{Deshpande:14}.
\begin{definition}
  \label{def:Deshpande}
  Let $C\coloneqq (x_1,r_1,...,x_n,r_n,x_1)$ be an elementary circuit of
  length $n\ge 1$ in the bipartite K{\"o}nig graph $\king$ and define
  \begin{equation}\label{eq:netstoich}
    \stout(C) \coloneqq
    \prod_{i=1}^n {(s^+_{x_{i+1}r_i} - s^-_{x_{i+1}r_i})} 
    \qquad\text{and}\qquad 
    \stin(C) \coloneqq
    \prod_{i=1}^n {(s^-_{x_ir_i} -  s^+_{x_ir_i})}
  \end{equation}
  Then we say that $C$ is \emph{autocatalytic} if $\stout(C)>\stin(C)$,
  \emph{drainable} if $\stout(C)<\stin(C)$, and \emph{neutral} if
  $\stout(C)=\stin(C)$.
\end{definition}
Neutral cycles are called s-cycles (stoichiometric cycles) in
\cite{Banaji:10,CraciunFeinberg:06}, where only networks without explicit
catalysts are considered.  Any elementary circuit $C$ is a fluffle in
$\king$ with the associated child-selection $\child$ naturally given by the
species and reactions in $C$ with the bijection $\kappa(x_i)=r_i$. We can
therefore express the products \eqref{eq:netstoich} of the net
stoichiometric coefficients of the reactants and products along the circuit
$C$ in the following form:
\begin{equation}
  \label{eq:stout} 
  \stout(C) = \prod_{i=1}^n \mathbf{S}[\child]_{i+1,i}
  \qquad\text{and}\qquad
  \stin(C) = (-1)^n \prod_{i=1}^n \mathbf{S}[\child]_{i,i}.
\end{equation}

Eq.~\ref{eq:stout} suggests that there is a close connection between the
determinant of the Metzler part $\Metzler{\SM[\child]}$ of a CS matrix
${\SM[\child]}$ and the elementary circuits in the bipartite König's graph
$\king$. Let us start from the definition
\begin{equation}\label{eq:leibniz}
  \det\mathbf{A}\coloneqq\sum_{\pi\in\mathsf{S}_{k}} \sgn(\pi)
  \prod_{j=1}^{k} \mathbf{A}_{j, \pi(j)}.
\end{equation}
Each permutation $\pi$, in turn, can be written as a product of cyclic
permutations $\chi$ on disjoint subsets of $\{1,\dots,k\}$. We say that a
cycle $\chi$ is \emph{contributing} if it comprises at least two elements
and $\prod_{j\in \chi} \mathbf{A}_{j,\chi(j)}\ne 0$, i.e., if
$\mathbf{A}_{j,\chi(j)}\ne 0$ for all $j\in \chi$.  Now, consider
$\mathbf{A}=\Metzler{\SM[\child]}$, we obtain the following lemma.
\begin{lemma}\label{lem:elemcircuitscont}
  An elementary circuit $C=(x_1,r_1,\dots,x_m,r_m)$ in the fluffle
  $\king(\child)$ defines a corresponding contributing permutation cycle
  $\chi=(x_1,x_2,\dots,x_m)$ in $\Metzler{\SM[\child]}$ if and only if (i)
  $C$ is not a digon and (ii) $C$ does not contain an edge $(r_i,x_{i+1})$
  with $s^+_{x_{i+1}r_{i}}\le s^-_{x_{i+1}r_{i}}$.
\end{lemma}
\begin{proof} Every elementary circuit $C$ in $\king(\child)$ defines a
permutation cycle $\chi$ which is contributing unless -- by definition --
$\chi$ is a singleton $\chi=(x)$, i.e., $C=(x,r)$, or
$\mathbf{S}_{\chi(x),\kappa(x)}=0$, for some $x$. The former conditions
translate to (i) and the second to (ii). In fact,
$\mathbf{S}_{\chi(x),\kappa(x)}=0$ occurs if and only if $\chi(x)$ is a
product of the reaction $\kappa(x)$, but not a net product,
i.e. $s^+_{\chi(x)\kappa(x)}\le s^-_{\chi(x)\kappa(x)}$. Conversely,
suppose $\chi$ is a contributing permutation cycle. In particular, $\chi$
comprises at least two elements and thus $x \ne \chi(x)$. We have that
$\Metzler{\SM[\child]}_{x,\chi(x)}\ne 0$ if and only if there is a (unique)
reaction $r$ such that $r=\kappa(x)$ and $\chi(x)$ is a product of $r$ with
$\mathbf{S}_{\chi(x),\kappa(x)}\ne0$. The latter condition is again
equivalent to (ii), while the condition on comprising at least two elements
is equivalent to (i).
\end{proof}

\begin{figure}
  \centering
  \begin{minipage}[c]{0.5\textwidth}
    \centering
    \begin{tikzpicture}
      \node[draw, circle] (x1) at (-4,0) {$x_1$};
      \node[draw, circle] (x2) at (0,-1) {$x_2$};
      \node[draw, circle] (x3) at (0,) {$x_3$};
      \node[draw, rectangle] (r1) at (-2,0) {$r_1$};
      \node[draw, rectangle] (r2) at (-2,-1) {$r_2$};
      \node[draw, rectangle] (r3) at (1,0) {$r_3$};
      \draw[->] (x1) to [bend left = 30] (r1);
      \draw[->] (r1) to [bend left = 30] (x1);
      \draw[->] (r1) to [bend left = 30] (x3);
      \draw[->] (x2) to (r2);
      \draw[->] (r2) to [bend left = 30] (x1);
      \draw[->] (x3) to [bend left = 30] (r3);
      \draw[->] (r3) to [bend left = 30] (x2);
    \end{tikzpicture}
  \end{minipage}%
  \hfill%
  \begin{minipage}[c]{0.5\textwidth}
    \centering
    \begin{tikzpicture}
      \node[draw, circle] (x1) at (0,0) {$x_1$};
      \node[draw, circle] (x2) at (4,0) {$x_2$};
      \node[draw, circle] (x3) at (2,2) {$x_3$};
      \node[draw, rectangle] (r1) at (0,2) {$r_1$};
      \node[draw, rectangle] (r2) at (2,0) {$r_2$};
      \node[draw, rectangle] (r3) at (4,2) {$r_3$};
      \draw[->] (x1) to  (r1);
      \draw[->] (r1) to  (x3);
      \draw[->] (x2) to (r2);
      \draw[->] (r2) to (x1);
      \draw[->] (r2) to  (x3);	      
      \draw[->] (x3) to (r3);
      \draw[->] (r3) to  (x2);
    \end{tikzpicture}
  \end{minipage}
  \caption{\textbf{Left}: Example of a fluffle $\king(\child)$ with a
    digon, associated to the RN \eqref{examplei}. and \textbf{Right:}
    Example of a fluffle $\king(\child)$ with a catalyst $x_1$ such
    that $s^+_{x_1r_2} \leq s^-_{x_1r_2}$, associated with the RN
    \eqref{exampleii}.}
  \label{fig:ExElCirc}
\end{figure}
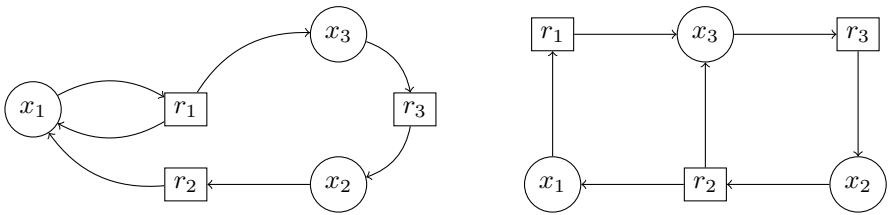

To see that conditions (i) and (ii) are actually important in RN with 
catalysts consider the following two examples:
\par\noindent
\begin{minipage}[c]{0.5\textwidth}
  \begin{equation}\label{examplei}
    \begin{split}
      x_1 + x_2 & \underset{r_1}{\longrightarrow} x_3 + x_1 	\\[0.5em]
      x_2 & \underset{r_2}{\longrightarrow} x_1 		\\[0.5em]
      x_3 & \underset{r_3}{\longrightarrow} x_2			\\[0.5em]
      \kappa_1(x_i) & = r_i, i=1,2,3 						\\[1.0em]			 
      \Metzler{\SM[\child_1]} &= \begin{pmatrix}
        0 & 1  &  0 \\ 
	0 & -1 &  1 \\  
	1 & 0  & -1 \end{pmatrix}
    \end{split}
  \end{equation}
\end{minipage}%
\begin{minipage}[c]{0.5\textwidth}
  \begin{equation}\label{exampleii}
    \begin{split}
      x_1 & \underset{r_1}{\longrightarrow} x_3  		\\[0.5em]
      2x_1 + x_2 & \underset{r_2}{\longrightarrow} x_1 +x_3	\\[0.5em]
      x_3 & \underset{r_3}{\longrightarrow} x_2			\\[0.5em]
      \kappa_2(x_i) & = r_i, i=1,2,3 						\\[1.0em]			 
      \Metzler{\SM[\child_2]} &= 
      \begin{pmatrix} 	
         -1 & 0 & 0 \\ 
	0  & -1 & 1 \\  
	1  &  1 & -1
      \end{pmatrix}
    \end{split}
  \end{equation}
\end{minipage}
\vspace*{1ex}
\par\noindent In each case the matrix $\Metzler{\SM[\child_{1/2}]}$ has one
contributing permutation cycle, \linebreak $\chi_1=(x_1,x_3,x_2)$ and
$\chi_2=(x_2,x_3)$, respectively. In contrast, the associated fluffle
$\king(\child_{1})$ as depicted in Fig.~\ref{fig:ExElCirc} (\textbf{left})
has two elementary circuits: the one represented in the CS-matrix
\linebreak $C^1_2=(x_1,r_1,x_3,r_3, x_2, r_2, x_1)$, i.e.,
$\chi=(x_1,x_3,x_2)$ and the additional digon $C^1_1=(x_1, r_1, x_1)$
corresponding to the trivial permutation cycle $\chi_1=(x_1)$. Similarily,
$\king(\child_{2})$ (Fig.~\ref{fig:ExElCirc} \textbf{right}) has two
elementary circuits, $C^1_1=(x_2,r_2,x_3,r_3)$, which corresponds to the
permutation cycle $\chi_2=(x_2,x_3)$, and
$C^2_2=(x_1,r_1,x_3,r_3, x_2, r_2, x_1)$ that is not reflected in
$\Metzler{\SM[\child_{2}]}$.

Both conditions (i) and (ii) in Lemma \ref{lem:elemcircuitscont} are void
in RNs without catalysts, and hence we recover the 1-1 correspondence
between elementary circuits in $\king(\child)$ and the contributing
circuits of $\Metzler{\SM[\child]}$ \cite[Lem.~16]{Golnik:25z}. In the
general case, Lemma \ref{lem:elemcircuitscont} motivates the following
definition.
\begin{definition}[Contributing circuits]\label{def:contributingcircuits}
  We say that an elementary circuit $C$ in $\king(\child)$ is
  \emph{contributing} if its corresponding permutation cycle is
  contributing, i.e.  if $|V(C)|\ge 4$ and $C$ does not contain an edge
  $(r_i,x_{i+1})$ with $s^+_{x_{i+1}r_{i}}\le s^-_{x_{i+1}r_{i}}$.
\end{definition}
Comparing the Leibniz formula Equ.~\eqref{eq:leibniz}
with a single permutation cycle and Def.~\ref{def:Deshpande} immediately
yields
\begin{lemma}
  Let $C$ be an induced elementary circuit in $\king$ and denote by
    $\child\coloneqq\child(C)$ the corresponding CS. Then
  \begin{equation}
    \det \mathbf{S}[\child] = 
    (-1)^{n-1} (\stout(C)-\stin(C)).
  \end{equation}
\end{lemma}

To accommodate back-edges in $\king(\child)$ introduced by
catalyzed reaction, we will say that a subgraph $H$ of $G$ is \emph{weakly
induced} if the underlying undirected graph of $H$ is an induced subgraph
of the underlying undirected subgraph of $G$. In other words, $H$ is an
induced subgraph $G$ up to ignoring back-edges in $G$. Note that in RNs
without explicit catalysis, induced and weakly induced subgraphs are the
same due to the absence of digons.

Let us now consider an elementary circuit $C$ that is weakly induced in
$\king(\child)$. We observe that $(-1)^{n-1} \det \mathbf{S}[\child] >0$ if
and only if $C$ is an \emph{autocatalytic} circuit in the sense of
Def.~\ref{def:Deshpande}, which by Prop.~\ref{prop:irMHu=sp} is equivalent
to $\mathbf{S}[\child]$ being semipositive, i.e., $\king(C)$ is
autocatalytic in the usual sense.  Moreover, the corresponding fluffle
$\king(\child)$ cannot contain a strictly smaller strong block, and thus
$C$ is an autocatalytic core. We can summarize these observations as
follows:
\begin{corollary}
  \label{cor:weaklyinduced} 
  Let $C$ be a weakly induced elementary circuit in $\king(\child)$. Then
  the following statements are equivalent:
  \begin{itemize}
  \item[(i)]   $C$ is an autocatalytic core.
  \item[(ii)]  $C$ is an autocatalytic elementary circuit in the sense of
    Def.~\ref{def:Deshpande}.
  \item[(iii)] $(-1)^{|C|-1} \det \mathbf{S}[\child(C)] >0$.
  \end{itemize}
\end{corollary}
Below, we will make use of the fact that the statements in particular also
hold for induced elementary circuits. We are now in the position to prove
the main result of this section:

\begin{theorem}
  \label{lem:reversibilityfull}
  Let $(X',R')$ an autocatalytic core with \emph{reversible extension}
  $(X',R'\cup\bar{R}')$.  If $(X',R'\cup\bar{R}')$ contains more than one
  autocatalytic core, then $(X',R')$ contains a drainable circuit.
\end{theorem}
\begin{proof}
  Since $(X',R')$ is an autocatalytic core, there is a unique associated CS
  \linebreak $\child=(X',R',\kappa)$ with CS matrix
  $\SM[\child]$. Moreover, $(X',R'\cup\bar{R}')$ contains at least one
  autocatalytic core, namely $(X',R')$. First, consider the special case
  $|X'|=|R'|=1$, i.e., $r$ is an autocatalytic reaction. One immediately
  checks that $\bar{r}$ is not autocatalytic and $(\{x\},\{r,\bar{r}\})$
  cannot contain another autocatalytic core. We therefore may assume
  $|X'|=|R'|\ge 2$.  Let $(\tilde{X}',\tilde{R}')$ be an autocatalytic core
  in the reversible extension $(X',R' \cup \bar{R}')$ with associated CS
  $\tilde{\child}=(\tilde{X}',\tilde{R}',\tilde{\kappa})$ and CS-matrix
  $\SM[\tilde{\child}]$ and assume $\tilde\child\ne\child$. Note that we have
  $\tilde{X}'\subseteq X'$ and thus $\tilde{R}'\not\subseteq R'$, since
  otherwise $(X',R')$ would not be an autocatalytic core. Thus there is
  $\bar{r}\in\bar{R}'\cap\tilde{R}'$.

  By Lemma~\ref{lem:square+kappa}, every reaction $r\in R'$ has exactly one
  reactant in $X'$, and hence every reverse reaction $\bar{r}\in\bar{R}'$
  has exactly one product. By Cor.~\ref{cor:diagonal} and
  Lemma~\ref{lem:Blokhuis4}, both $\SM[\child]$ and $\SM[\tilde{\child}]$
  are irreducible Metzler matrices with non-positive diagonal. Thus
  reaction $\bar{r}\in\tilde{R}'\cap\bar{R}'$ has \emph{at most one}
  product $x$ with $s^+_{x\bar{r}}>0$ and at least one net reactant $y$
  with $\SM_{y\bar{r}}<0$. Moreover, the reactant $y\ne x$ satisfies
  $s^+_{y\bar{r}}=0$ because $\king(\child)$ is an induced fluffle and
  digons, i.e., explicit catalysis, can only be of the form
  $(x,\kappa(x),x)$ which prevents $s^-_{yr}=s^+_{y\bar{r}}>0$ since $x\ne
  y$. By Lemma~\ref{lem:square+kappa}, the CS matrix of an autocatalytic
  core is invertible and hence cannot contain two linearly dependent
  columns.  Thus, if $\bar{r}\in\tilde{R}'$ then $r\notin \tilde{R}'$.

  Now consider a reaction $\bar{r}_1\in\bar{R'}\cap\tilde{R}'$. It has a
  unique product $x_1\coloneqq \kappa^{-1}(r_1)$ and by virtue of the CS
  $\tilde\child$, $\kappa^{-1}(r_1)$ is the unique reactant of the reaction
  $\bar{r}_2\coloneqq\tilde\kappa(\kappa^{-1}(r_1))\in\tilde{R}'$. Since
  $x_1$ is a substrate only of $r_1$ in $R'$ and $r_1\notin\tilde{R}'$ we
  have $\bar{r}_2\in\bar{R}'$. Repeating the argument we observe that
  $\tilde{R}'\subseteq\bar{R'}$. Moreover, each reaction
  $\bar{r}\in\tilde{R}'$ has only a single reactant
  $\tilde\kappa^{-1}(\bar{r})\in \tilde{X}'\coloneqq
  \tilde\kappa^{-1}(\tilde{R}')$ and a single product $\kappa^{-1}(r)\in
  \tilde{X}'\subseteq X'$. Since $C\coloneqq
  \king[\tilde{X}'\cup\tilde{R}']$ must be an induced fluffle by
  Lemma~\ref{lem:fufflesufficient} and Cor.~\ref{cor:ac-inducedfluffle},
  and every vertex has indegree and outdegree $1$, we conclude that $C$ is
  an induced elementary circuit.

  First assume $\tilde{X}'=X'$. Then $\tilde{R}'=\bar{R}'$ and thus every
  reaction $r\in R'$ also has only a single reactant and a single product,
  and hence $\king[X'\cup R']$ forms the reverse $\bar{C}$ of the
  elementary circuit $C$. Since $(X',R')$ is autocatalytic with have
  $(-1)^{n-1}\det\SM[\child]=\stout(\bar{C})-\stin(\bar{C})>0$ and thus
  $(-1)^{n-1}\det\SM[\tilde\child]=\stout(C)-\stin(C) =
  \stin(\bar{C})-\stout(\bar{C})<0$, and thus $C$ is a drainable cycle and
  hence not autocatalytic. Therefore we have $\tilde{X}'\subsetneq X'$ and
  $\tilde{R}'\subsetneq\bar{R'}$. Since the elementary circuit $C$ is, by
  assumption, an autocatalytic core, and its reverse $\bar{C}$ is, by
  construction, a subgraph of the fluffle $\king[X'\cup R']$, we have
  $(-1)^{n-1}\det\SM[\tilde\child]=\stout(C)-\stin(C)>0$ and thus
  $(-1)^{n-1}\det\SM[\child]=\stout(\bar{C})-\stin(\bar{C})<0$, i.e,
  $\bar{C}$ is a drainable circuit within the autocatalytic core $(X',R')$.
\end{proof}

The last paragraph in the proof of Thm.~\ref{lem:reversibilityfull} in
particular recovers the irreversibility result in the SI of \cite{Kosc:25}
in the following form:
\begin{corollary}
  Let $(X',R')$ be an autocatalytic core and $R''\subset
  R'\cup\bar{R'}$. Then $(X',R'')$ is an autocatalytic core if and only
  if $R''=R'$.
\end{corollary}

Drainable cycles within the autocatalytic core $(X',R')$, on the other
hand, may give rise to autocatalytic cores $(X'',R'')$ on a strict
subset of the reverse reactions, i.e., $X''\subsetneq X'$ and
$R''\subsetneq\bar{R}'$. To see this, consider the following examples:
\medskip
\par\noindent
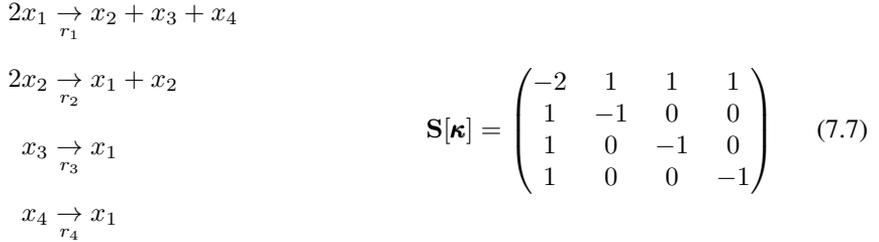
\begin{minipage}[c]{0.45\textwidth}
\smallskip\par\noindent
\centering
  \begin{equation*}\label{eq:revSys1}
    \begin{split}
      2x_1 & \underset{r_1}{\rightarrow} x_2+ x_3+ x_4 \\[0.5em]
      2x_2 & \underset{r_2}{\rightarrow} x_1 + x_2 \\[0.5em]
      x_3 & \underset{r_3}{\rightarrow} x_1 \\[0.5em]
      x_4 & \underset{r_4}{\rightarrow} x_1 \\[0.5em]
    \end{split}
  \end{equation*}
\end{minipage}%
\hfill%
\begin{minipage}[c]{0.5\textwidth}
\smallskip\par\noindent
\centering
  \begin{equation}
    \SM[\child] = \begin{pmatrix}
      -2 & 1 & 1 & 1\\
      1 & -1 & 0 & 0\\
      1 & 0 & -1 & 0\\
      1 & 0 & 0 & -1\\
    \end{pmatrix}
  \end{equation}
\end{minipage}
\smallskip\par\noindent
\smallskip\par\noindent with CS $\child=(\{x_i \ \vert \ i=1,\ldots, 4\}, \{r_i
\ \vert \ i=1,\ldots, 4\}, \kappa), \kappa(x_i)\coloneqq r_i,
i=1,\ldots, 4$. 

\smallskip\par\noindent Then the reversible extension contains three
additional autocatalytic cores: the explicit autocatalytic reaction $x_2
\underset{\bar{r_2}}{\rightarrow} 2x_2$ as well as
$\tilde{\child}_{3/4}\coloneqq (\{x_1,x_{3/4}\},
\{\bar{r}_1,\bar{r}_{3/4}\}, \tilde{\kappa}_{3/4})$ with
$\tilde{\kappa}_{3/4}(x_1)=\bar{r}_{3/4}$ and
$\tilde{\kappa}_{3/4}(x_{3/4})=\bar{r}_1$, which both yield the following
autocatalytic CS matrix and core:
\begin{equation}
  \SM[\child_{3/4}] = \begin{pmatrix}
    -1 & 2 \\
    1 & -1 \\
  \end{pmatrix}.
\end{equation}
Note that all three additional cores in the reversible extension are indeed
drainable circuits in the autocatalytic core $(X',R')$, as stated in
Thm.~\ref{lem:reversibilityfull}. Moreover, explicit catalysis only occurs
in the case of a digon.  We formalize this last observation in the next
corollary.
\begin{corollary}\label{cor:DigOrFree}
  Let $(X',R')$ be an autocatalytic core and let $(X'',R'')$ be an
  autocatalytic core in the reversible extension $(X',R'\cup\bar{R}')$ with
  $X''\subsetneq X'$ (and thus $R''\subsetneq\bar{R}')$. Then $(X'',R'')$
  is either a single autocatalytic reaction or it does not contain an
  explicit catalyzed reaction $\bar{r}\in R''$.
\end{corollary}
\begin{proof}
  Assume that $\bar{r}\in R''$ is an explicit catalyzed reaction. Then
  there is $x\in X''$ such that $\bar{r} =\tilde{\kappa}(x)$, i.e., $x$ is
  the unique reactant of $\bar{r}$, and $(x,\bar{r},x)$ is a digon in
  $\king(\tilde{\child})$. However, $x=\kappa^{-1}(r)$ is also the unique
  product of $\bar{r}$. If $|X''|=|R''|=1$ then $\tilde{\mathcal{A}}$ is an
  autocatalytic reaction. Otherwise, $|X''|=|R''|\geq 2$. Thus, there is a
  reaction $\bar{r}^*$ such that there is a path $(x,\bar{r}_1, x_1,
  \ldots, \bar{r}^*)$. However, $\bar{r}$ has only one product, i.e., $x$,
  hence $\bar{r}_1\neq \bar{r}$. Thus, $x$ is a reactant for another
  reaction in $R''$, which by Lemma~\ref{lem:square+kappa} contradicts that
  the assumption that $(X'',R'')$ is an autocatalytic core.
\end{proof}

If all reactions are considered to be reversible, the definition of an
autocatalytic core would not include autocatalytic cores $(X',R')$ that
contain a smaller autocatalytic core in $(X',\bar{R}')$ because one would
identify $R'$ and $\bar{R}'$. In such a setting, for example,
\cite{Blokhuis:20} provided a complete classification of autocatalytic cores
in five types, which is then not valid in the presence of irreversible
reactions. It may be useful, therefore, to distinguish autocatalytic cores
that do not contain any other core in its reversible extension.  
\begin{definition}
  An autocatalytic core (CS-core) is \emph{hard} if its reversible
  extension does not contain another autocatalytic core.
\end{definition}
In particular, therefore, Thm.~\ref{lem:reversibilityfull} states:
\begin{corollary}
  An autocatalytic core is hard if and only if it does not contain a
  drainable elementary circuit.
\end{corollary}

\paragraph{Minimal Autocatalytic Subsystems (MAS).} MAS were introduced in
\cite{Gagrani:24}, and they are defined with yet another version of minimality. Let
us write $X(R')\subseteq X$ for the set of entities that appear as
reactants or products in the set of reactions. Thus $X(R')$ correspond to
the non-zero rows in submatrix of $\SM$ formed by columns in $R'$.
\begin{definition} \cite{Gagrani:24}
  A sub-RN $(X(R'),R')$ is a \emph{minimal autocatalytic subsystem} (MAS)
  if $(X',R')$ is autocatalytic for some $X'\subseteq X(R')$ and
  there is no proper subset $R''\subsetneq R'$ such that $(X'', R'')$ is autocatalytic.
\end{definition}

An MAS contains at least one autocatalytic core. However, there may be
  autocatalytic cores $(X',R')$ and $(X'',R')$ with the same set $R'$ of
  reactions and $X'\ne X''$. As an example consider
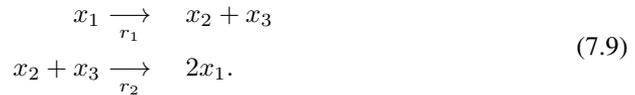
\begin{equation}
  \begin{split}
    x_1  	  &\underset{r_1}{\longrightarrow}\quad x_2 + x_3\\
    x_2 + x_3     &\underset{r_2}{\longrightarrow}\quad 2x_1.\\
  \end{split}
\end{equation}
In particular, $R'\coloneqq \{r_1,r_2\}$ and $X(R')\coloneqq
\{x_1,x_2,x_3\}$ gives rise to two autocatalytic cores, encoded by the CS
$\child_{2/3}\coloneqq (\{x_1, x_{2/3}\}, \{r_1,r_2\}, \kappa_{2/3})$ with
$\kappa_{2/3}(x_1)\coloneqq r_1$ and $\kappa_{2/3}(x_{2/3})\coloneqq r_2$.
Minimality w.r.t.\ the number of reactions implies that an MAS does not
contain more reactions than an autocatalytic core. This relationship
remains unchanged if explicit catalysts are allowed.

\begin{figure}
  \centering
  \begin{tikzpicture}
    \node[draw, ellipse, minimum width = 2.0cm, minimum height = 1.1cm, label={[font=\footnotesize, align = center, yshift=-10pt] center: {$\mathcal{M}$}}] (a) at (0,-0.5) {};
    \node[draw, ellipse, minimum width = 3.5cm, minimum height = 2.6cm, label={[font=\footnotesize, align=center, yshift=20pt] 
	center:{Autocatalytic cores}}] (b) at (0,0.25) {};
    \node[draw, ellipse, minimum width = 2.5cm, minimum height = 1.25cm, label={[font=\footnotesize, align=center, yshift=10] 
	center:{Hard cores}}] (c) at (0,0) {};
    \node[draw, ellipse, minimum width = 5cm, minimum height = 3.6cm,
      line width=1.4pt,
      label={[font=\footnotesize, align=center, yshift=32pt] center:{Autocatalytic CS cores}}] (d) at (0,0.75) {};
  \end{tikzpicture}
  
  \caption{Summary of the subset relationship between different notions of
    minimal autocatalytic systems with or without explicit
    catalysis. $\mathcal{M}$ denotes the set of autocatalytic cores induced
    by MAS (minimal autocatalytic subsystems) \emph{sensu}
    \cite{Gagrani:24}.}
    \label{fig:CoreRelWithoutCatalysis}
\end{figure}
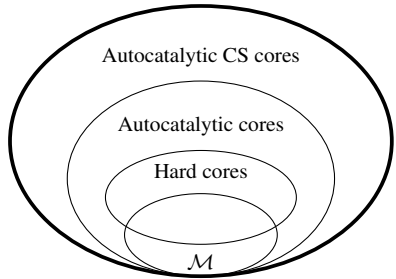

As expected, an MAS contains both hard and non-hard cores. For example, in
the system:
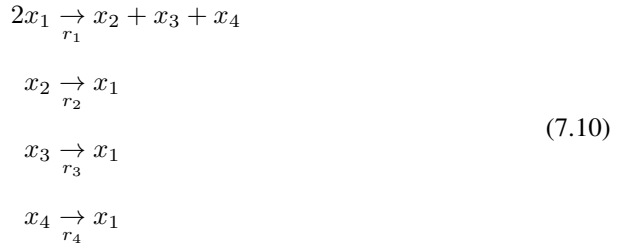
\begin{equation}
  \begin{split}
    2x_1 & \underset{r_1}{\rightarrow} x_2 + x_3+ x_4 \\[0.5em]
    x_2 & \underset{r_2}{\rightarrow} x_1 \\[0.5em]
    x_3 & \underset{r_3}{\rightarrow} x_1 \\[0.5em]
    x_4 & \underset{r_4}{\rightarrow} x_1 \\[0.5em]
  \end{split}
\end{equation}
with $R'=\{r_i \ | \ i=1,\ldots, 4\}$ and $X(R')=\{x_i \ | \ i=1,\ldots,
4\}$, $(X(R'),R')$ is an autocatalytic core, but not a hard core, because
$\child=(\{x_1,x_2\}, \{\bar{r_1}, \bar{r_2}\}, \kappa)$, with
$\kappa(x_1)= \bar{r_2}, \kappa(x_2)= \bar{r_1}$ is an autocatalytic core,
and in particular a hard core, in the \emph{reversible extension} of
$(X(R'), R')$. The MAS $(X(\{\bar{r_1}, \bar{r_2}\}), \{\bar{r_1},
\bar{r_2}\})$, thus contains only one autocatalytic core, i.e., a hard one.

In general, a set of reactions may contain more autocatalytic cores than
MAS, as shown by the following example taken from the Appendix of
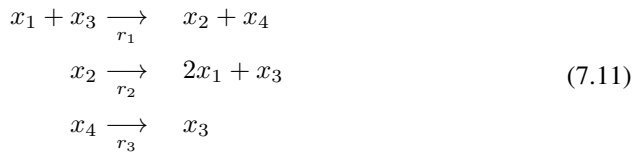
\cite{Golnik:25z}:
\begin{equation}
    \begin{split}
      x_1 + x_3 &\underset{r_1}{\longrightarrow}\quad x_2 + x_4\\
      x_2       &\underset{r_2}{\longrightarrow}\quad 2x_1 + x_3\\
      x_4       &\underset{r_3}{\longrightarrow}\quad x_3\\
    \end{split}
\end{equation}

It contains two sub-RNs, $(X',R')=(\{x_1,x_2\},\{r_1,r_2\})$ with CS
$\kappa'(x_1)=r_1$, $\kappa'(x_2)=r_2$ and
$(X'',R'')=(\{x_2,x_3,x_4\},\{r_1,r_2,r_3\}$ with $\kappa(x_2)=r_2$,
$\kappa(x_3)=r_1$ and $\kappa(x_4)=r_3$ both of which are autocatalytic.
Both $(X',R')$ and $(X'',R'')$ are autocatalytic cores since
$X'\not\subseteq X''$. However, since $R'\subsetneq R''$, only $(X(R'),R')$
forms an MAS. The set $\mathcal{M}$ of autocatalytic cores ``induced'' by
the set of all MAS in a given RN, therefore, is a subset of the set of all
autocatalytic cores. In addition, since $(X'',R'')$ is a hard core, not all
hard cores are induced by an MAS.  This discussion is summarized in
Fig.~\ref{fig:CoreRelWithoutCatalysis}.

\section{Reaction Networks with Unit Stoichiometry}

Reaction networks that model cell signalling by means of protein
phosphorylation typically do not use stoichiometric coefficients other than
$0$ and $1$. Since they are also easier to analyze than general reaction
networks, they have received considerable attention in the literature, see
\cite{Jiao:25} and references therein. Unit stoichiometry is also
frequently observed in metabolic pathways, including glycolysis,
pentose-phosphate-pathway, or tricarboxylic acid cycle. 

\begin{definition}
  An RN $(X,R)$ has \emph{unit stoichiometry} if $s^{\pm}_{xr}\in\{0,1\}$.
\end{definition}
It follows immediately that $\mathbf{S}_{xr}\in\{-1,0,+1\}$. We again allow
explicit catalysis. Moreover, we note that the results in this section also
pertain to sub-RNs $(X',R')$ with unit stoichiometry that may be embedded
in large networks with unrestricted stoichiometry. We first note that
handling explicit catalysis becomes particularly simple in autocatalytic
cores with unit stoichiometry:

\begin{lemma}\label{lem:contrelemLeibniz}
  Let $(X',R')$ be an autocatalytic core with CS-matrix $\SM[\child]$ in a
  RN with unit-stoichiometry. Then one of the
  following alternatives hold true.
  \begin{enumerate}[(i)]
  \item All diagonal entries of $\SM[\child]$ are nonzero and there 
    are at least two contributing circuits $C_1$ and $C_2$ with a common
    vertex.
  \item The diagonal of $\SM[\child]$ contains a zero entry and there is at
    least one contributing elementary circuit $C^*$ that contains vertices
    $x\in X'$ with $\SM[\child]_{xx}= 0$.
    \end{enumerate}
\end{lemma}
\begin{proof}
  By Lemma~\ref{lem:descartes} we have $\sgn\det\SM[\child]=(-1)^{n-1}$. In
  case (i), we consider the following alternatives: Assume first that there
  are no contributing elementary circuits; then the only contributing
  permutation to the Leibniz formula is the identity one, and thus
  $\sgn\det\SM[\child]=(-1)^k$, a contradiction. If there is only a single
  cycle, then the Leibniz formula yields
  $\det\SM[\child]=(-1)^{k-1}+(-1)^k=0$, again a contradiction.  Finally,
  assume that no two contributing elementary circuits share a vertex and
  that there are at least $n\ge 2$ elementary circuits. For the symmetric
  group $\mathsf{S}_k$ we denote with {$\mathfrak{C}$} the set of all
  contributing permutations, i.e., all $\pi$ for which $\prod_{i=1}^k
  \SM[\child]_{i,\pi(i)}\neq 0$. Each $\pi\in\mathfrak{C}$ is a product of
  $c(\pi)$ contributing cycles $\chi$ with length $\ell(\chi)$. With this
  notation, we compute
  \begin{equation}
    \begin{split}
      \det\SM[\child]
      &=\sum_{\pi\in\mathsf{S}_k} \sgn(\pi)
      \prod_{j=1}^{k}\SM[\child]_{j, \pi(j)}=
      \sum_{\pi\in\mathfrak{C}} \sgn(\pi)
      \prod_{j=1}^{k}\SM[\child]_{j, \pi(j)}\\[0.5em]
      & = \sum_{\pi\in\mathfrak{C}}
      \underbrace{(-1)^{\sum_{\chi \in \pi}(\ell(\chi)-1)}}_{=\sgn(\pi)}
      \underbrace{(-1)^{k-\sum_{\chi\in\mathfrak{C}}\ell(\chi)} \cdot 1}_{\prod_{j=1}^{k} \SM[\child]_{j, \pi(j)}} \\[0.5em]
      &= (-1)^{k} \sum_{\pi\in\mathfrak{C}} (-1)^{-c(\pi)}
      = (-1)^{k} \sum_{j=0}^{N} (-1)^j = 0. 
    \end{split}
  \end{equation}
  The penultimate equality follows from the fact that every contributing
  permutation $\pi$ is is a combination of $j$ of the $N$ pairwisely
  non-overlapping contributing cycles. The binomial theorem finally
  ensures that the sum vanishes. Hence we again obtain a contradiction to
  Lemma \ref{lem:descartes}.
  \par\noindent In case (ii), we assume there are no contributing
  elementary circuits that contain all such vertices $x$ with
  $\SM[\child]_{xx}= 0$, then each summand of the Leibniz formula would be
  identically zero, again a contradiction.
\end{proof}

In the following, it will be convenient to consider networks primarily from
the perspective of not necessarily induced subgraphs of $\king$ rather than
as sub-matrices of the stoichiometric matrix. To this end, we will need the
following notion:
\begin{definition}[Spanning fluffle]
A fluffle $G=(V_1\cup V_2, E_1\cup E_2)$ is a \emph{spanning fluffle} of a
subnetwork $(X',R')$ if $V_1(G)=X'$ and $V_2(G)=R'$.
\end{definition}
Moreover, for a child-selection $\pmb{\kappa}$ and a not necessarily
spanning subgraph $G$ of $\king(\child)$ such that $E_1(G)=E_1(\child)\cap
X(G)\times R(G)$ we define the \emph{fluffle-part} matrix $\mathbf{A}(G)$
\begin{equation}    
\mathbf{A}(G)_{ij} = \begin{cases}
\SM[\child]_{ij}\quad&\text{if $i=j$ or $i\neq j$ and $(\kappa(x_i),x_j)\in {E_2(G)}$}\\
0 \quad &\text{otherwise}.
\end{cases}
\end{equation}

In \cite{Golnik:25z} we briefly discussed an interesting structural
dichotomy distinguishing two types of autocatalytic sub-RNs, depending on
whether or not they can be interpreted as an amplification mechanism for a
subset of ``central'' species. It seems useful to investigate the
corresponding fluffles without restricting them to autocatalytic sub-RNs.
\begin{definition}[Centralized fluffle]
  A fluffle $G$ is \emph{centralized} if there is a \emph{center} $x^*\in
  X'$ such that all elementary circuits of length at least $4$ in $G$ pass
  through $x^*$.
\end{definition}
Since circuits in bipartite graphs are always of even length, the
  minimal length requirement simply excludes digons, i.e., elementary
  circuits arising from explicit catalysis. The reason for disregarding
  these digons is that only the net stoichiometry matters for contributing
  permutation cycles.

Then, it is not difficult to see that this definition is consistent with the definition of
centralized child-selections in \cite{Golnik:25z}.   Centers are
feedback vertex sets of size of $1$. Graphs with centers are ``almost
acyclic'' since deleting a center obviously leaves a directed acyclic
graph. Moreover, if there is more than one center, then all centers lie on
a common elementary circuit.  Digraphs with a center have appeared, e.g.,
in \cite{Kosaraju:77,Thomassen:85}. Digraphs with centers can be
constructed by ear decompositions with the property that every ``return
path'' from the terminal to initial vertex of an ear runs through all
centers.  If initial and terminal vertices of ears are in addition
restricted to reaction and metabolite vertices, respectively, one obtains a
centralized fluffle.  An induced elementary circuit $C$ in $\king$ is a
special case of a centralized fluffle.

\begin{lemma}
  \label{lem:unitdet} 
  Let $G$ be a centralized fluffle in an RN with unit-stoichiometry. Let $N$
  be the number of contributing cycles that contains all vertices $x\in X'$
  with $\mathbf{A}(G)_{xx}=0$, and let $N'=0$ if $\mathbf{A}(G)$ has a
  non-zero diagonal entry, and $N'=1$ otherwise. Then
  $(-1)^{n-1}\det\mathbf{A}(G) = N-N'$.
\end{lemma}
\begin{proof}
  If $\child$ is centralized in $x^*$ then all contributing circuits run
  through $x^*$ and thus only the corresponding permutation cycles and the
  singleton permutation cycles corresponding to the diagonal elements can
  be non-zero. Thus a permutation $\pi$ yields a non-zero contribution only
  if it is of the form $\pi=\chi\circ(u_{k+1})\circ\dots\circ(u_n)$ where
  $\chi$ is contributing cycle of length $k$ and the remaining permutations
  are singletons. We therefore compute via the Leibniz formula:
  \begin{equation*}
    \begin{split}
      \sgn(\pi)\underbrace{\prod_{x\in X'}\mathbf{A}(G)_{x,\pi(x)}}_{=1, \forall x\in \chi} & =
      \sgn(\chi) \prod_{j= k +1}^n \mathbf{A}(G)_{u_ju_j} =\\
      &=(-1)^{k-1} \begin{cases} (-1)^{n-k} \\ 0 \end{cases} =  
  \begin{cases} (-1)^{n-1} \\ 0 \end{cases}
	  \end{split}
  \end{equation*}
  We obtain a non-zero value if and only if all of the $x\in X'$ that
  appear as singletons, or equivalently, are not contained in the
  permutation cycle have a nonzero (diagonal) entry,
  i.e. $\mathbf{A}(G)_{x,x}\neq 0$, i.e. do not correspond to digons.
  Similarly, the diagonal contribution corresponding to the identity
  permutation, $\prod_{x\in X'}\mathbf{A}(G)_{x,x}$ is either $(-1)^n$ or
  $0$, depending on whether or not there is a zero diagonal entry. Thus we
  have $\Det\mathbf{A}(G) = (-1)^{n-1}N + (-1)^n N'$, and the statement
  follows.
\end{proof}
We remark that Lemma~\ref{lem:unitdet} and its proof remain correct if
``centralized'' is replaced by the weaker condition that any two
contributing circuits share at least one vertex. In this case, no
permutation contains two contributing circuits, and thus the contributions
of the Leibniz formula have the same form as in the proof of
Lemma~\ref{lem:unitdet}. This suggests that fluffles without
vertex-disjoint circuits could be indeed of interest for our
purposes. General digraphs without vertex-disjoint circuits have been
discussed e.g.\ in \cite{Thomassen:87}.

\begin{theorem}\label{thm:generation}
  For every autocatalytic core $(X',R')$ with unit stoichiometry, there is
  a spanning fluffle $G$ of the form $G=C\cup P$, where $C$ is an
  elementary circuit and $P$ is an ear connecting a reaction-vertex $r\in
  R'$ to a substrate-vertex $x\in X'$.
\end{theorem}
\begin{proof}
  Let $(X',R')$ be an autocatalytic core with CS-matrix
  $\SM[\pmb{\kappa}]$.  We consider two complementary cases.
  \par\noindent\emph{Case (i).  All diagonal entries of $\SM[\pmb{\kappa}]$
  are nonzero.}  Lemma \ref{lem:contrelemLeibniz} implies that there exists
  two contributing elementary circuits $C_1$ and $C_2$ that share at least
  one vertex. We consider the fluffle $G \coloneqq C_1 \cup C_2$, which is
  then centralized.
  Lemma~\ref{lem:unitdet} implies that $(-1)^{|V({G})|-1}
  \det\mathbf{A}(G)=2-1>0$ and in particular $\mathbf{A}(G)$ is
  Hurwitz-unstable.  Thus, by Prop.~\ref{prop:irMHu=sp} $\mathbf{A}(G)$ is
  semipositive.  Moreover, $E_1(G)$ defines a sub-CS $\child'$ of $\child$
  on $\tilde{X}'=X'\cap V(G)$ and $\tilde{R}'=R'\cap V(G)$ and we have, by
  construction, $\mathbf{A}(G)\le \SM[\child']$ (element-wise), i.e.,
  $\child'$ on $(\tilde{X}',\tilde{R}')$ is an autocatalytic CS.  Thus $G$
  is a spanning fluffle for $\king(\child)$, since $\king(\child)$ is an
  autocatalytic core.
  \par\noindent\emph{Case (ii). At least one diagonal
  entry of $\SM[\pmb{\kappa}]$ is zero.}  Lemma \ref{lem:contrelemLeibniz}
  implies that there exists one contributing elementary circuit $C^*$ that
  contains all such vertices $x$ with $\SM[\child]_{xx}= 0$.  Consider the
  fluffle $G^*=C^*$, which is clearly centralized.  Lemma~\ref{lem:unitdet}
  implies that $(-1)^{|V(G^*)|-1}\Det\mathbf{A}(G^*)=1>0$ and in particular
  $\mathbf{A}(G^*)$ is Hurwitz-unstable.  Thus, by
  Prop.~\ref{prop:irMHu=sp} $\mathbf{A}(G_1)$ is semipositive.  Moreover,
  $E_1(G^*)$ defines a sub-CS $\tilde{\child}$ of $\child$ on $\tilde{X}'=X'\cap
  V(G)$ and $\tilde{R}'=R'\cap V(G)$ and we have, by construction,
  $\mathbf{A}(G^*)\le \SM[\tilde{\child}]$, i.e., $\tilde{\child}$ on
  $(\tilde{X}',\tilde{R}')$ is an autocatalytic CS.  Thus $G^*$ is a
  spanning fluffle for $\king(\child)$, since $\king(\child)$ is an
  autocatalytic core.
\end{proof}

Next we observe that Thm.~\ref{lem:reversibilityfull} can be strengthened
in the case of unit stoichiometry:
\begin{corollary}
  \label{cor:rev-unit}
  If $(X',R')$ is an autocatalytic core with unit stoichiometry, then
   it is hard, i.e., there is no other autocatalytic core in its
  reversible extension $(X',R'\cup\bar{R}')$.
\end{corollary}
\begin{proof}
  Assume $(X',R'\cup\bar{R}')$ contains another autocatalytic core
  $(\tilde{X}',\tilde{R}')$. By the proof of
  Thm.~\ref{lem:reversibilityfull}, we have, $\tilde{X}' \subsetneq X'$,
  $\tilde{R}'\subsetneq \bar{R}'$, and thus $(\tilde{X}',\tilde{R}')$ is an
  induced elementary circuit $C$. Thm.~\ref{thm:generation} implies that
  $C$ cannot be autocatalytic with unit stoichiometries.
\end{proof}

It may be the case that the spanning fluffle $G$ is actually just a single
elementary circuit $C$. This leads us to a ``classification'' corollary of
Thm.~\ref{thm:generation}:
\begin{corollary}[Classification]\label{cor:classification}
In RNs with unit stoichiometry, there is only one class of autocatalytic 
cores for which there is \emph{no} spanning fluffle $G$ that is 
a \emph{single} elementary circuit $G\neq C$, and this is
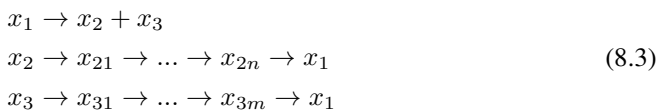
\begin{equation}\label{eq:2circuits}
  \begin{split}
    &x_1\rightarrow x_2 + x_3\\
    &x_2\rightarrow x_{21} \rightarrow ... \rightarrow x_{2n} \rightarrow x_1\\
    &x_3\rightarrow x_{31} \rightarrow ... \rightarrow x_{3m} \rightarrow x_1\\
  \end{split}
\end{equation}
with $x_1 \neq x_2 \neq x_3$ and $n, m\ge0$. In particular, such
autocatalytic cores are without explicit catalysis and correspond to 
type III of Blokhuis classification \cite{Blokhuis:20}.
\end{corollary}
\begin{proof}
Let $(X',R')$ be an autocatalytic core with associated CS-matrix
$\SM[\child]$. Again, the dicotomy expressed in Lemma
\ref{lem:contrelemLeibniz} and consequently in the proof of
Thm.~\ref{thm:generation} holds and for the case (ii), i.e. when
$\SM[\child]$ has a zero diagonal entry, we get directly that any
autocatalytic core has a spanning fluffle $G=C$, which is a single
elementary circuit. In case (i), we apply Thm.~\ref{lem:reversibilityfull}
and we obtain again $(X',R')$ as a unique autocatalytic core for the
reversible setting, since with unit stoichiometry no drainable circuits
\textit{sensu} Def.~\ref{def:Deshpande} are possible. In this case, we can
reduce to the case without explicit catalysis since minimality of $(X',R')$
is determined in terms of all sub-matrices of $\SM[\child]$, which have a
negative diagonal.  In fact, unit stoichiometry,
Lemma~\ref{lem:cs-no-ends}~(ii), and Cor.~\ref{cor:acore-uniqueCS} directly
exclude the presence of explicit catalysis in $(X',R')$. We can then
consider the $(X',R')$ as a full network without catalysis, and apply the
classification theorem from \cite{Blokhuis:20}. By quick inspection of the
five types described in \cite{Blokhuis:20}, omitted here, we conclude that
\eqref{eq:2circuits} is the only possible example that does not allow for a
spanning fluffle consisting of a single elementary circuit.
\end{proof}

\section{Summary and concluding remarks} 

Much of the literature on chemical reaction networks excludes explicit
  catalysts, i.e., reactions in which reactants also appear as
  products. This is, in particular, the case for previous studies concerned with
  autocatalytic substructures. Here, we have shown that key features of
  autocatalytic cores established in
  \cite{Blokhuis:20,Vassena:24a,Golnik:25z} are retained in the more
  general setting, but there are indeed subtle differences. These are
  summarized in Table~\ref{tab:Comparison}.

\begin{table}[tb]
\caption{Comparison of an autocatalytic core $(X',R')$ without and with explicit catalysis.}
\label{tab:Comparison}
\centering
\begin{tabular}{|c|c|c|} 
  \hline
  &									& 							\\[-0.75em]
  & Without Explicit Catalysis & With Explicit Catalysis \\[0.5em]
  \hline
  &									& 										\\[-0.75em]
  SR -			& $\forall x \in X'\; \exists ! r\in R': s^-_{xr}>0$
  & $\forall x \in X' \;\exists !
  r\in R': s^-_{xr}>0$\\
  relations		& $\forall x \in X' \;\exists r\in R': s^+_{xr}>0$ 	& $\forall x \in X' \;\exists r\in R': s^+_{xr}-s^-_{xr}>0 $ \\[0.25em]
  & \textbf{\cite{Blokhuis:20, Vassena:24a}} 		& \textbf{(Lemma~\ref{lem:cs-no-ends})} \\[0.25em]
  
  \hline
  &									\multicolumn{2}{|c|}{}							\\[-0.75em]
  CS 				& \multicolumn{2}{c|}{Unique CS $\child=(X', R',\kappa)$}\\[0.25em]
  & \multicolumn{2}{c|}{\textbf{\cite{Vassena:24a}} \hspace{102pt} \textbf{(Cor.~\ref{cor:acore-uniqueCS})}}\\[0.25em]
  \hline					&						&					\\[-0.75em]
  & \emph{Metzler} 			& \emph{Metzler} \textbf{(Cor.~\ref{cor:diagonal})}\\[0.25em]
  & negative diagonal 			& nonpositive diagonal \textbf{(Cor.~\ref{cor:diagonal})} \\[0.25em]
  $\SM[\child]$	     	& \emph{semipositive} 		& \emph{semipositive} \textbf{\cite{Vassena:24a}} \\[0.25em]
  & \emph{irreducible} 			& \emph{irreducible} \textbf{(Lemma~\ref{lem:Blokhuis4})}\\[0.25em]
  & \textbf{\cite{Vassena:24a}} 	& \\[0.25em]
  \hline 
  &									\multicolumn{2}{|c|}{} 				\\[-0.75em]
  Bipartite   					& \multicolumn{2}{|c|}{$\king(\child)=\king[X'\cup R']$ \emph{(induced Fluffle)}}\\[0.25em]
  graph					& \multicolumn{2}{|c|}{\textbf{\cite{Golnik:25z}} \hspace{84pt} \textbf{(Corollary~\ref{cor:ac-inducedfluffle})}}\\[0.25em]
  \hline
  &									& 							\\[-0.75em]
  Digons					& Absent 	& Possible as $(x,\kappa(x),x)$ 	\\[0.5em]	
  \hline
  &									& 							\\[-0.75em]
  Reversible 				& Elementary circuit   					& Elementary circuit (or Digon)		\\
  cores					& \textbf{Thm.~\ref{lem:reversibilityfull}, Cor.~\ref{cor:DigOrFree}} & \textbf{Thm.~\ref{lem:reversibilityfull}, Cor.~\ref{cor:DigOrFree}} \\[0.5em]
  \hline
  &									& 							\\[-0.75em]
  & Elementary circuit  & Elementary circuit \\	
  Unity $\mathbf{S}$			& + RM-Ear \textbf{(Thm.~\ref{thm:generation})} & + RM-Chord \textbf{(Thm.~\ref{thm:generation})}
  \\
[0.5em]
  & Only hard cores \textbf{(Cor.~\ref{cor:rev-unit})} & Only hard cores \textbf{(Cor.~\ref{cor:rev-unit})} \\[0.5em]
  \hline  
\end{tabular}
\end{table}

First, autocatalytic cores may appear in the form $(\{x\},\{r\})$ of a
single explicitly autocatalytic reaction $r$ acting on a single entity $x$,
with $s^+_{xr}-s^-_{xr}>0$. In contrast, without catalysts, autocatalytic
cores require at least two species and two reactions \cite{Blokhuis:20}.
In line with previous results \cite{Vassena:24a}, autocatalytic cores in
RNs with explicit catalysis have a unique CS
(Cor.~\ref{cor:acore-uniqueCS}).  Moreover, CS matrices of autocatalytic
cores are again Metzler matrices. However, their diagonal entries are not
necessarily strictly negative; they may vanish for the catalyst $x$ of
$\kappa(x)$, or even be positive in the special case of the core being a
single species and an autocatalytic reaction, where the resulting
one-dimensional matrix satisfies $\SM[\child] = s^+_{xr} - s^-_{xr} > 0$.

Explicitly catalyzed reactions have a specific structure representation in
the bipartite K\"onig graph, as they emerge as digons of the form
$(y,\kappa(x), y)$. In autocatalytic cores, only digons appear where $x$
equals $y$, while for larger autocatalytic subsystems, also multiple
catalysts are theoretically possible. The bipartite K\"onig representation
of autocatalytic cores as \emph{fluffles} is thereby conserved in the
presence of explicit catalysis, since the unique CS directly translates
into a unique perfect matching of the reactant-to-reaction edges.

A useful practical implication of the mathematical results reported here is
that the algorithmic approaches to enumerating autocatalytic cores
developed for RNs without explicit catalysis \cite{Golnik:25z,Golnik:26q}
remain valid in RNs with explicit catalysis. \texttt{Autogatito} \cite{Golnik:26q}, 
a python package determining all autocatalytic cores via enumeration of 
metabolite-to-reaction chord-free elementary circuits, in particular, lists digons 
in a preprocessing step, and allows digons whose reactant-to-reaction edges 
are located along the circuit, while its reaction-to-product edges are tolerated 
chords. It thus accounts for the presence of explicitly catalyzed reactions in the RN.
\texttt{Autogato} \cite{Golnik:25z}, on the other hand, enumerates elementary 
circuits, regardless of the presence of chords of any type in the induced subgraph. 
This algorithm detects not only autocatalytic cores with explicitly catalyzed reactions, 
but also subsystems associated with autocatalytic CS matrices with irreducible 
Metzler part where digons are of the form $(y,\kappa(x), y), x\neq y$. 

Larger \emph{fluffles} can be enumerated efficiently via the superposition
of elementary circuits by taking advantage of an equivalence relationship
defined by the reactant-to-reaction edges \cite{Golnik:25z}. In general, some 
cores can only be enumerated by the superposition of at least 3 elementary 
circuits \cite{Golnik:26q}. However, for the case of unit stoichiometry, which 
applies to cell signalling networks as well as to many but not all metabolic 
pathways, autocatalytic cores have spanning fluffles that are composed of 
a single elementary circuit and an attached MR-ear. If the core contains an 
explicitly catalyzed reaction, such an MR-ear is a simple MR-chord. This 
has especially significant consequences for computational complexity, 
and allows for terminating the enumeration process conveniently after all 
suitable superpositions of two elementary circuits have been tested.

The two types described in the last paragraph reflect the classification
system for RNs, consisting only of reversible reactions without explicit
catalysis \cite{Blokhuis:20}. Although chemical reactions are always
reversible, thermodynamic constraints often render reactions effectively
irreversible since backward rates are many orders of magnitude
slower. Realistic models of chemical RNs are thus usually composed of
irreversible and reversible reactions. While autocatalytic cores with
spanning fluffles that are composed of a single elementary circuit and a
simple MR-chord recover type II, IV, and V, elementary circuits with
attached MR-ears that are not chords, recover type III from
\cite{Blokhuis:20}. A single elementary circuit (type I) with unit
stoichiometries, however, is \textit{a priori} not autocatalytic.

Extending autocatalytic cores with arbitrary stoichiometries by all of
their reverse reactions may give rise to additional autocatalytic
cores. These additional autocatalytic cores, however, are always elementary
circuits and consist of reversed reactions only. Moreover, they are either
digons corresponding to explicit autocatalytic reactions or longer
elementary circuits that are then free of explicit catalysis. The latter
conform to type I in the classification of \cite{Blokhuis:20}. In RNs with
unit stoichiometry, moreover, the reverse of any autocatalytic core never
contains another autocatalytic core.

To account for reversible RNs, i.e., comprising for each reaction also its
reverse, we introduced the notion of \emph{hard autocatalytic cores} for
autocatalytic cores for which a subset of reverse reactions on the same
entities is never autocatalytic.  The hard autocatalytic cores in the
irreversible setting are in 1-1 correspondence to autocatalytic cores in
the setting where all reactions are considered reversible. Note that the
classification theorem of \cite{Blokhuis:20} pertains to the latter
framework only.

Revisiting the relationship between unstable cores and autocatalytic cores
\cite{Vassena:24a}, we clarified the relationships between cores minimal
w.r.t.\ to the subset relation on $X\cup R$ and CS cores minimal
w.r.t.\ the restriction of child-selections. While autocatalytic cores in
\cite{Blokhuis:20} are defined w.r.t.\ the former, the unstable cores of
\cite{Vassena:24a} are defined w.r.t.\ the latter. We therefore introduced
here a notion of \emph{autocatalytic CS cores}. They are associated with
autocatalytic CS matrices with irreducible Metzler part. Thus, the approach
taken in \cite{Golnik:25z} in particular also enumerates all autocatalytic
CS cores, which is, however, computationally very
expensive. Therefore, it remains an open question of practical interest
whether they can be found more efficiently from the set of all
autocatalytic cores.

In contrast to autocatalytic cores, whose bipartite K\"onig representation
are induced fluffles, for autocatalytic CS cores, only $\king(\child)$ is a
fluffle, while the induced subgraph may contain additional edges. In
particular, autocatalytic CS cores have reactions with multiple reactants
within the core. This directly translates into their matrix
representation. In the case of RNs without explicit catalysis, the set of
autocatalytic CS cores with a Metzler CS matrix coincides with the set of
autocatalytic cores. In the case of RNs with explicit catalysis, however,
autocatalytic CS cores that are not autocatalytic cores could also exhibit
a Metzler CS matrix (Fig.~\ref{fig:CoreRelationships}). 

\subsection*{Acknowledgements}

This work was supported in part by the Novo Nordisk Foundation (grant
NNF21OC0066551 ``MATOMICS''.  Research in the Stadler lab is support by the
BMBF (Germany) through DAAD project 57616814 (SECAI, School of Embedded
Composite AI), and jointly with SMWK (Saxony) through the \emph{Center for
Scalable Data Analytics and Artificial Intelligence Dresden/Leipzig}
(SCADS24B).

\bibliographystyle{adamjoucc}
\bibliography{fullCRN}
\label{sect:biblio}

\end{document}